\documentclass[11pt,twoside,a4paper]{article}

\usepackage[utf8]{inputenc}
\usepackage[T1]{fontenc}
\usepackage{amsmath,amssymb,dsfont}
\numberwithin{equation}{section}
\usepackage{microtype}
\usepackage{lmodern}
\usepackage{graphicx,tikz,pgfplots}
\usepackage{tikz-3dplot}
\usetikzlibrary{arrows,shapes,fadings,intersections,decorations.pathreplacing,positioning}
\graphicspath{{images/}}
\pgfplotsset{compat=newest}
\usepackage[hyperref,amsmath,thmmarks]{ntheorem}
\usepackage{aliascnt}
\usepackage[a4paper,centering,bindingoffset=0cm,marginpar=2cm,margin=2.5cm]{geometry}
\usepackage[pagestyles]{titlesec}
\usepackage[font=footnotesize,format=plain,labelfont=sc,textfont=sl,width=0.75\textwidth,labelsep=period]{caption}
\usepackage{comment}
\usepackage{siunitx}
\usepackage{mathtools}

\usetikzlibrary{quotes}
\usetikzlibrary{angles}
\usetikzlibrary{babel}

\usetikzlibrary{decorations.pathreplacing}
\usetikzlibrary{decorations.markings}

\usepackage{mathabx}
\usetikzlibrary{external}
\tikzset{
    >=stealth',
    punkt/.style={
           rectangle,
           rounded corners,
           draw=black, very thick,
           text width=6.5em,
           minimum height=2em,
           text centered},
    pil/.style={
           ->,
           thick,
           shorten <=2pt,
           shorten >=2pt,}
}
\tikzstyle{block} = [rectangle, rounded corners, minimum width= 3cm, minimum height=1cm, text centered, draw=black, fill=blue!20,]

\tikzstyle{invisbleblock} = [minimum width= 3em, minimum height=1cm, text centered, ]

\tikzstyle{decision} = [diamond, minimum width=3cm, minimum height=1cm, text centered, draw=black, fill=orange!20]

\tikzstyle{arrow} = [thick,->,>=stealth]
\tikzstyle{line} = [thick,-]

\usepackage{dsfont,bbm}

\usepackage[boxed]{algorithm2e}

\usepackage{enumerate}

\usepackage[labelfont=normalfont,justification=RaggedRight]{subcaption}
\usepackage{multirow}
\usepackage{tabulary}

\usepackage[backend=biber,maxnames=10,backref=true,hyperref=true,giveninits=true,safeinputenc]{biblatex}
\DefineBibliographyStrings{english}{%
	backrefpage = {cited on page},
	backrefpages = {cited on pages},
}

\title{	Generalized Fourier Diffraction Theorem and Filtered Backpropagation for Tomographic Reconstruction}
\author{Clemens Kirisits$^{1,4}$\\
	{\footnotesize\href{mailto:clemens.kirisits@univie.ac.at}{clemens.kirisits@univie.ac.at}}  
	\and Michael Quellmalz$^2$\\
	{\footnotesize\href{mailto:quellmalz@math.tu-berlin.de}{quellmalz@math.tu-berlin.de}}
	\and Eric Setterqvist$^3$\\
	{\footnotesize\href{mailto:eric.setterqvist@santa-anna.se}{eric.setterqvist@santa-anna.se}}  
}

\date{\today}

\DeclareFieldFormat[report]{title}{``#1''}
\DeclareFieldFormat[book]{title}{``#1''}
\AtEveryBibitem{\clearfield{url}}
\AtEveryBibitem{\clearfield{note}}
\AtEveryBibitem{\clearfield{issn}}

\titleformat{\section}[block]{\large\sc\filcenter}{\thesection.}{0.5ex}{}[]
\titleformat{\subsection}[runin]{\bf}{\thesubsection.}{0.5ex}{}[.]
\titleformat{\subsubsection}[runin]{\bf}{\thesubsubsection.}{0.5ex}{}[.] 

\usepackage[pdftex,colorlinks=true,linkcolor=blue,citecolor=green!50!black,urlcolor=blue,bookmarks=true,bookmarksnumbered=true,pdfusetitle]{hyperref}
\hypersetup
{
    pdfauthor={C. Kirisits, M. Quellmalz, E. Setterqvist},
}

\newpagestyle{headers}
{
	\headrule
	\sethead[\footnotesize\thepage][\footnotesize\sc C.~Kirisits, M.~Quellmalz, E.~Setterqvist][]{}{\footnotesize\sc Generalized Fourier diffraction theorem
  }{\footnotesize\thepage}
	\setfoot{}{}{}
}
\newtheorem{lemma}{Lemma}[section]

\newaliascnt{proposition}{lemma}

\aliascntresetthe{proposition}

\newaliascnt{corollary}{lemma}
\newtheorem{corollary}[corollary]{Corollary}
\aliascntresetthe{corollary}

\newaliascnt{theorem}{lemma}
\newtheorem{theorem}[theorem]{Theorem}
\aliascntresetthe{theorem}

\theorembodyfont{\normalfont}
\newaliascnt{definition}{lemma}

\aliascntresetthe{definition}

\newaliascnt{assumption}{lemma}

\aliascntresetthe{assumption}

\newaliascnt{notation}{lemma}

\aliascntresetthe{notation}

\newaliascnt{example}{lemma}
\newtheorem{example}[example]{Example}
\aliascntresetthe{example}

\newaliascnt{experiment}{lemma}

\aliascntresetthe{experiment}

\newaliascnt{remark}{lemma}
\newtheorem{remark}[remark]{Remark}
\aliascntresetthe{remark}

\theoremstyle{nonumberplain}
\theoremseparator{:}
\theoremheaderfont{\normalfont\itshape}

\theoremsymbol{\ensuremath{\square}}
\newtheorem{proof}{Proof}

\newcommand{\N}{\mathds{N}}
\newcommand{\R}{\mathds{R}}
\newcommand{\C}{\mathds{C}}
\renewcommand{\S}{\mathbb{S}}

\newcommand{\abs}[1]{\left|#1\right|}
\newcommand{\norm}[1]{\left\|#1\right\|}

\renewcommand{\Re}{\operatorname{Re}}

\newcommand{\e}{\mathrm e}
\let\ii\i
\renewcommand{\i}{\mathrm i}

\newcommand{\B}{\mathcal{B}}

\newcommand{\dd}{\, \mathrm{d} }

\newcommand{\ba}{\mathbf{a}}
\newcommand{\bd}{\mathbf{d}}
\newcommand{\bn}{\mathbf{n}}
\newcommand{\br}{{\mathbf r}}
\newcommand{\bs}{{\bf s}}
\newcommand{\bx}{{\bf x}}
\newcommand{\by}{{\bf y}}
\newcommand{\bk}{{\bf k}}

\newcommand{\bt}{{\bf t}}
\newcommand{\bu}{{\bf u}}
\newcommand{\bv}{{\bf v}}
\newcommand{\be}{{\bf e}}
\newcommand{\bg}{{\bf g}}
\newcommand{\bz}{{\bf z}}
\newcommand{\bo}{{\bf 0}}
\newcommand{\bh}{{\bf h}}
\newcommand{\bp}{{\bf p}}

\newcommand{\ui}{u^{\mathrm{inc}}}

\newcommand{\sym}{\mathrm{sym}}

\newcommand{\utot}{u^{\mathrm{tot}}}

\newcommand{\us}{u^{\mathrm{sca}}}

\newcommand{\rM}{r_{\mathrm{M}}}

\newcommand{\sgn}{\operatorname{sgn}}
\newcommand{\supp}{\operatorname{supp}}

\newcommand{\ktran}{\mathcal{F}}

\begin{document}

\maketitle
\thispagestyle{empty}
\begin{center}
	\parbox[t]{11em}{\footnotesize
		\hspace*{-1ex}$^1$Faculty of Mathematics\\
		University of Vienna\\
		Oskar-Morgenstern-Platz 1\\
		A-1090 Vienna, Austria}
	\hspace{1em}
	\parbox[t]{11em}{\footnotesize
		\hspace*{-1ex}$^2$Institute of Mathematics\\
		Technical University Berlin\\
		Straße des 17. Juni 136 \\
		D-10623 Berlin, Germany}
	\hspace{1em}
	\parbox[t]{15em}{\footnotesize
		\hspace*{-1ex}$^3$Santa Anna IT Research Institute\\
		SE-58183 Link\"oping, Sweden} 
\end{center}
\begin{center}
	\parbox[t]{19em}{\footnotesize
		\hspace*{-1ex}$^4$Christian Doppler Laboratory for\\
		\hspace*{1em}Mathematical Modeling and Simulation of\\
		\hspace*{1em}Next-Generation Ultrasound Devices (MaMSi)\\
		Oskar-Morgenstern-Platz 1\\
		A-1090 Vienna, Austria}
\end{center}

\begin{abstract}
This paper concerns diffraction-tomographic reconstruction of an object characterized by its scattering potential. 
We establish a rigorous generalization of the Fourier diffraction theorem in arbitrary dimension, giving a precise relation in the Fourier domain between measurements of the scattered wave and reconstructions of the scattering potential. 
With this theorem at hand, Fourier coverages for different experimental setups are investigated taking into account parameters such as object orientation, direction of incidence and frequency of illumination. 
Allowing for simultaneous and discontinuous variation of these parameters, a general filtered backpropagation formula is derived resulting in an explicit approximation of the scattering potential for a large class of experimental setups.
\end{abstract}

\section{Introduction} \label{sec:exp}

\paragraph{The Helmholtz equation.}
We consider an inverse source problem for the Helmholtz equation
	\begin{equation}\label{eq:hh-intro}
	-(\Delta + k_0^2) u = g \quad \text{in } \R^d,
	\end{equation}
	where $k_0$ is a positive constant and $g$ an integrable function with compact support. Given measurements of the unique outgoing solution $u$, the aim is to reconstruct $g$. Our first result connects the Fourier transform of $g$ with that of $u$ restricted to a hyperplane and may be seen as a generalization of the well-known Fourier diffraction theorem \cite{KakSla01,NatWue01,Wol69}. Let $\tilde \ktran \colon \mathcal{S}'(\R^d) \to \mathcal{S}'(\R^d)$ be the partial Fourier transform along the first $(d-1)$ coordinates. 
	Then, $\tilde \ktran u$ is a locally integrable function given by
	\begin{equation}\label{eq:fdt-intro}
	\tilde \ktran u(\bx, \rM) = \sqrt{\frac{\pi}{2}} \frac{\i}{\kappa} \left(\e^{\i \kappa \rM} \ktran (g^-)(\bx,\kappa) + \e^{-\i \kappa \rM} \ktran (g^+)(\bx,-\kappa) \right), \quad \bx \in \R^{d-1},\ \rM\in \R,
	\end{equation}
	where $\kappa = \kappa(\bx)$ is the principal square root of $k_0^2 - \abs{\bx}^2$ and $\ktran$ is the Fourier transform on $\R^d$, continued analytically to $\C^d$ for $\abs{\bx}^2 > k_0^2$. The functions $g^+$ and $g^-$ are defined by $g^\pm(\br) = g(\br)$ if $r_d \gtrless \rM$ and $g^\pm(\br) = 0$ otherwise. Assuming that $u$ is measured on the hyperplane $\{\br \in \R^d : r_d = \rM \}$, then \eqref{eq:fdt-intro} relates the spatial frequency components of the data to those of $g^+$ and $g^-$.
	
\autoref{thm:fourierdiffraction} where \eqref{eq:fdt-intro} is proved generalizes \cite[Thm.\ 3.1]{KirQueRitSchSet21}, which covers only the case $d=3$ and requires the compactly supported inhomogeneity $g$ to belong to $L^p(\R^3)$ with $p>1$. The stricter condition on $g$ in \cite[Thm.\ 3.1]{KirQueRitSchSet21} is a consequence of the dependency on integrability estimates of $u$, see \autoref{rem:Lp-estimates}, in its proof. In contrast, the proof of \autoref{thm:fourierdiffraction} relies on a certain continuity of the map $g \mapsto u$ established in \autoref{thm:Rcontinuous} for which the weaker assumption $g\in L^1(\R^d)$ is sufficient.
There exist many other formulations of the Fourier diffraction theorem in $d\in\{2,3\}$ such as \cite{BorWol19,Dev82,KakSla01,MulSchGuc15,NatWue01,Wol69} that only consider the plane wave case discussed in the following paragraph.
Among these references, \cite{NatWue01} imposes the stronger assumption that $g\in L^2(\R^d)$, while the others do not state the conditions explicitly.
In general, the derivations in the above cited references are based on the plane wave decomposition of the Green's function of the Helmholtz equation. The proof of \autoref{thm:fourierdiffraction} on the other hand is done in the framework of tempered distributions with the major component being, besides \autoref{thm:Rcontinuous}, the partial Fourier transform of the Green's function given in \autoref{thm:G1}.

\paragraph{Diffraction tomography.}
The Helmholtz equation \eqref{eq:hh-intro} arises as a model for the scattering of time-harmonic waves $U(\br,t) = \operatorname{Re} (\utot(\br)\e^{-\i \omega t})$ from a bounded inhomogeneity. Assuming that the wave motion is caused by an incident field $\ui$ propagating through a homogeneous background until it meets the scatterer, a common model for the resulting scattered field $\us = \utot - \ui$ is
\begin{equation}
	-(\Delta + k^2_0) \us = k_0^2f\utot\quad \text{in } \R^{ d}. \label{eq:reduced}
\end{equation}
In addition, $\us$ satisfies the Sommerfeld radiation condition, see \cite[Chap.\ 8.1]{ColKre19}. In this context, $k_0>0$ is the wave number of the incident field $\ui$, which is assumed to solve $\Delta \ui + k_0^2\ui = 0$, and the normalized scattering potential $f$ is given by
\begin{equation}\label{eq:f}
  f(\br) = \frac{n(\br)^2}{n_0^2} - 1, \quad \br \in \R^d,
\end{equation}
where $n$ is the refractive index. Outside the bounded inhomogeneity, we have $n(\br) \equiv n_0$, so $f$ is compactly supported. In general $f$ has a nonzero imaginary part in order to allow for absorption.

Two common simplifications of this scattering model are the Born and the Rytov approximation, each leading to an equation of the form \eqref{eq:hh-intro}. The first-order Born approximation neglects the term $k_0^2 f \us$ on the right-hand side of \eqref{eq:reduced} and reads
\begin{equation}\label{eq:hhborn}
	-(\Delta + k^2_0) \us = k_0^2f\ui.
\end{equation}
The first-order Rytov approximation is based on the ansatz $\utot = \ui \e^{\varphi}$ with a complex phase function $\varphi$ and leads to
\begin{equation}\label{eq:hhrytov}
	-(\Delta + k^2_0) (\ui \varphi) = k_0^2f\ui.
\end{equation}
Further details on the derivation and validity of these approximations can be found, for instance, in \cite[Chap.\ 6]{KakSla01} and also \cite{FauKirQueSchSet23}. Within this framework, the inverse problem of diffraction tomography, see \cite{Dev12,KakSla01,NatWue01}, can be formulated as follows: Given knowledge of the incident field $\ui$ as well as measurements of the scattered wave on a hyperplane in $\R^d$, recover the scattering potential $f$ based on \eqref{eq:hhborn} or \eqref{eq:hhrytov}.

A key result underlying diffraction tomography, the Fourier diffraction theorem, can be obtained as a special case of \eqref{eq:fdt-intro} by setting $g = k_0^2 f \ui$, see \eqref{eq:hhborn}, and assuming that (i) measurements are taken outside the support of $f$ and (ii) the incident field is a plane wave $\ui (\br) = \e^{\i k_0 \br \cdot \bs}$ propagating in direction $\bs \in \mathbb{S}^{d-1}$. Then \eqref{eq:fdt-intro} becomes the more familiar
\begin{equation}\label{eq:fdt}
	\tilde \ktran u (\bx, \rM) = \sqrt{\frac\pi2}\, \frac{ \i k_0^2 \e^{\pm\i \kappa \rM}}{\kappa} \mathcal{F}f \left( (\bx,\pm\kappa) - k_0\bs \right) \quad \text{if } \rM \gtrless r_d \text{ for all } \br \in \supp f.
\end{equation}
Note that \eqref{eq:fdt-intro} can be used to obtain a relation between scattered wave and scattering potential even when $\ui$ is not a plane wave, see \autoref{rem:applicability}.

In applications, mainly the spatial frequencies $\bx \in \R^{d-1}$ with $\abs{\bx} < k_0$ are relevant. This provides information about $\ktran f$ at the points $(\bx, \pm \kappa (\bx))-k_0\bs $ on a sphere in $\R^d$ with radius $k_0$ and center $-k_0\bs$. Knowledge of $\ktran f$ on this set only is not sufficient for a reasonable recovery of $f$. Therefore, reconstruction algorithms in diffraction tomography crucially rely on data collection strategies gathering additional information by varying one or more of the following parameters of the experiment: (i) the direction $\bs$ of the incident wave, (ii) the orientation of the scatterer described by a rotation matrix $R\in SO(d)$, or (iii) the wave number $k_0$ of $\ui$. In \autoref{sec:coverage}, we investigate how changes in each of these parameters (plus additional ones which are shown to be ineffective) influence the coverage in Fourier space. A general experiment, where all parameters are allowed to change simultaneously depending on time $t\in[0,L]$, leads in $\R^d$ to the Fourier coverage
\begin{equation*}
	\mathcal{Y} = \left\{ R(t) \left( (\bx,\pm \kappa(\bx,t)) - k_0(t)\, \bs(t) \right) \in \R^d : 0\le t \le L, \, |\bx| < k_0(t) \right\},
\end{equation*}
where $\kappa$ depends on $t$ through $k_0$.

\paragraph{Filtered backpropagation.}
Filtered backpropagation, as pioneered in \cite{Dev82}, provides an explicit reconstruction formula for
\begin{equation*}
	f_\mathcal{Y}(\br) \coloneqq
  (2\pi)^{-\frac{d}{2}} \int_{\mathcal{Y}} \ktran f(\by) \e^{\i \br \cdot \by} \dd \by
\end{equation*}
for every $\br \in \R^d$. The idea is to first apply the change of coordinates
\begin{equation*}
	\by = T(\bx,t) \coloneqq R(t) \left( (\bx, \kappa(\bx,t)) - k_0(t)\, \bs(t) \right),
\end{equation*}
and then use the Fourier diffraction theorem \eqref{eq:fdt} to replace the spatial frequency components of $f$ with those of the measurements $u$. One issue with this approach is that $T$ is far from injective in general. Therefore, in order to correctly extend filtered backpropagation formulas to the general setting proposed here, one has to account for the lack of injectivity by means of the Banach indicatrix
\begin{equation*}
	\operatorname{Card}\left(T^{-1}(\by)\right),
\end{equation*}
where $\operatorname{Card}$ denotes the counting measure. While the Banach indicatrix can be difficult to determine in general, we suggest a numerical procedure for estimating it.

Another issue, not only with filtered backpropagation but with diffraction tomographic methods in general, is the missing cone problem, cf.\ \cite{krauze2020optical,LeeKimOhPar21,LimLeeJinShiLee15,SunDas11}. This is the observation that for many experimental setups the Fourier coverage has significant cone-like holes close to the origin, see \autoref{fig:rot_wave-2d} or \autoref{fig:cover3}, for example. Consequently, a considerable portion of the low spatial frequencies of $f$ is not available for reconstruction, thus leading to poor results. The missing cone problem can be overcome, for instance, by consecutively rotating the object around more than one axis or by subsequently illuminating from more than one direction during rotation. In order to correctly incorporate the resulting measurements into a single backpropagation formula, we allow the functions $R(t)$, $\bs(t)$ and $k_0(t)$, and hence $T$, to have jump discontinuities. We establish the corresponding backpropagation formula allowing for noninjective and discontinuous $T$ in \autoref{thm:reconstruction}. Subsequently, we present an improvement of this formula for real-valued $f$ that exploits the conjugate symmetry of $\hat f$ and ensures that the reconstruction $f_{\mathcal{Y}}$ is real-valued as well. In general, it enlarges the Fourier coverage while reducing the amount of data required to achieve a certain coverage. Numerical tests show that the new backpropagation method can provide a reconstruction quality similar to the inverse NDFT method while being faster. The speed advantage becomes especially relevant when the reconstruction is used inside an iterative method such as for phase retrieval, see \cite{BeiQue22,BeiQue23}, where the Banach indicatrix needs to be computed only once.

\paragraph{Outline.}
This article is organized as follows. After reviewing certain results concerning the well-posedness of the forward problem $g \mapsto u$ associated to \eqref{eq:hh-intro} in \autoref{sec:hh}, we prove the generalized Fourier diffraction theorem in \autoref{sec:back}. \autoref{sec:coverage} is devoted to the systematic study of Fourier coverages resulting from various experimental setups. A universal filtered backpropagation formula together with some extensions and special cases is derived in \autoref{sec:backprop}. Finally, the discretization is discussed in \autoref{sec:discretization}, where we also present an estimation method of the Banach indicatrix. Numerical experiments are performed in \autoref{sec:numerics}.
\section{The Helmholtz equation in $\R^d$} \label{sec:hh}

In this section we recall results concerning outgoing solutions to the Helmholtz equation
\begin{equation}\label{eq:HH}
	-(\Delta + k_0^2) u(\br) = g(\br), \quad \br \in \R^d,
\end{equation}
for compactly supported $g$. A solution $u$ of \eqref{eq:HH} is \emph{outgoing,} if it satisfies the \emph{Sommerfeld radiation condition}  
\begin{equation}\label{eq:src}
	\lim_{r \to \infty} r^{\frac{d-1}{2}} \left( \frac{\partial u}{\partial r} - \i k_0 u \right) = 0
\end{equation}
uniformly in $\bs$, where $\br = r \bs$, $r =\abs{\br}$, and $\partial/\partial r$ denotes the radial derivative.
The significance of the Sommerfeld radiation condition is twofold. First, it characterizes outgoing waves. That is, if $u$ satisfies \eqref{eq:HH} and \eqref{eq:src}, then $U(\br,t) = \Re \left( u(\br) \e^{-\i \omega t} \right)$, $\omega > 0$, physically corresponds to a wave propagating away from the scatterer \cite[Chap.\ IV,§5]{CouHil62}. Second, the Sommerfeld radiation condition ensures uniqueness for \eqref{eq:HH}, see \autoref{thm:uniqueness}. Due to the hypoellipticity of $\Delta + k_0^2$, every distributional solution $u$ of \eqref{eq:HH} is smooth on $\R^d \setminus \supp g$, so that the differentiability requirement implicit in \eqref{eq:src} is always met for compactly supported $g$, see \cite{Fol95}.

We denote the space of test functions by $\mathcal{D}(\R^d)$, which consists of all compactly supported smooth functions,
and the space of distributions by $\mathcal{D}'(\R^d)$.
Furthermore, we will need the Schwartz space $\mathcal{S}(\R^d)$ of rapidly decreasing, smooth functions and the space of tempered distributions $\mathcal{S}'(\R^d)$.

An outgoing fundamental solution of the $d$-dimensional Helmholtz operator $-\Delta - k_0^2$ is given by 
\begin{equation}\label{eq:G}
	G(\br) = \frac{\i}{4} \left( \frac{k_0}{2 \pi |\br|} \right)^\frac{d-2}{2} H_{\frac{d-2}{2}}^{(1)} (k_0 \abs{\br}),
\end{equation}
where $H^{(1)}_a$ is the Hankel function of the first kind and order $a$. See \cite[Chap.\ 9]{McL00} for a derivation of \eqref{eq:G}. The function $G$ is also known as {Green's function} for the Helmholtz equation. Note that $G \in C^\infty (\R^d \setminus \{ 0\})$. Moreover, the limiting forms
\begin{equation}\label{eq:Hankel-limit-small}
\begin{aligned}
	H_a^{(1)}(z) \sim
		\begin{cases}
			-\frac{\i}{\pi} \Gamma(a) \left( \frac{z}{2} \right)^{-a}, & \operatorname{Re} a > 0, \\
			\frac{2\i}{\pi} \log (z), & a= 0 ,
		\end{cases}
\end{aligned}
\qquad \text{for} \quad z \to 0,
\end{equation}
and
\begin{align}\label{eq:Hankel-limit-large}
	H_a^{(1)}(z) \sim \sqrt{\frac{2}{\pi z}} \e^{\i\left( z - \frac{a \pi}{2} - \frac{\pi}{4} \right)}, \qquad \text{for} \quad z \to \infty,
\end{align}
imply that $G$ belongs to $L^1_{\mathrm{loc}}(\R^d) \cap \mathcal{S}'(\R^d)$ and that
\begin{equation}\label{eq:growth-G}
	G(\br) = \mathcal{O}\left(\left|\br\right|^{\frac{1-d}{2}}\right), \qquad \text{for} \quad  \left|\br\right|\rightarrow\infty,
\end{equation}
see \cite[\href{https://dlmf.nist.gov/10.2.E5}{(10.2.5)},\href{https://dlmf.nist.gov/10.7.E2}{(10.7.2)},\href{https://dlmf.nist.gov/10.7.E7}{(10.7.7)}]{DLMF}. Concerning tomographic reconstructions in the two- and three-dimensional space, notable special cases of \eqref{eq:G} are
\begin{equation}\label{eq:G123}
	G(\br) =
	\begin{dcases}
		\frac{ \i \e^{\i k_0 |\br|}}{2 k_0}, & d = 1, \\
		\frac{\i}{4}  H_{0}^{(1)} (k_0 \abs{\br}), & d = 2, \\
		\frac{ \e^{\i k_0 |\br|}}{4 \pi |\br|}, & d = 3.
	\end{dcases}
\end{equation}

The following theorem shows that outgoing solutions of \eqref{eq:HH} are unique in $\mathcal D'(\R^d)$ and that, in particular, $G$ is unique. Related results are \cite{Rel43,Som12}. See also \cite[Chap.\ IV,§5]{CouHil62} or \cite[Thm.\ 9.11]{McL00}.

\begin{theorem}\label{thm:uniqueness} For every $g \in \mathcal{D}'(\R^d)$ with compact support, there is at most one outgoing solution of \eqref{eq:HH}.
\end{theorem}
\begin{proof}
Suppose $u\in \mathcal{D}'(\R^d)$ is an outgoing solution of \eqref{eq:HH}. Every other solution of \eqref{eq:HH} can be written as $u+v$ where $\Delta v + k_0^2v = 0$ on $\R^d$. But $u+v$ can only be outgoing, if $v$ is. It now follows that $v$ must vanish identically from Green's formula, also known as Green's third identity,
\begin{equation} \label{eq:Greens-formula}
	v(\br) = \int_{\partial B} v \nabla G_\br \cdot \bn - G_\br \nabla v \cdot \bn \dd s,
\end{equation}
where $\br \in \R^d$ is arbitrary, $B$ is a closed ball not containing $\br$, $G_\br$ is a shorthand for $G(\br - \cdot)$ and $\bn$ is the outward pointing unit normal. See \cite[Thm.\ 2.5]{ColKre19} for $d=3$ or \cite[Thms.\ 7.12, 9.6]{McL00} for a more general formulation. Applying Green's second identity to \eqref{eq:Greens-formula} and exploiting the fact that both $v$ and $G_\br$ solve the homogeneous Helmholtz equation in $B$ shows that $v(\br) = 0$.
\end{proof}

\begin{theorem}\label{thm:src}
For every $g \in L^1(\R^d)$ with compact support, $u=g*G\in L^1_{\mathrm{loc}}(\R^d)\cap\mathcal{S}'(\R^d)$ is the unique distributional solution of the Helmholtz equation \eqref{eq:HH} that satisfies the Sommerfeld radiation condition \eqref{eq:src}.
\end{theorem}
\begin{proof}
First, $u=g*G$ solves \eqref{eq:HH} in the distributional sense, because $G$ is a fundamental solution of $-\Delta - k_0^2$.
Next, we show that $u \in L^1_{\mathrm{loc}}(\R^d)\cap\mathcal{S}'(\R^d)$. Take a compact set $K\subset\R^d$ and a bounded set $B\supset\supp g$. Using the Fubini--Tonelli theorem, local integrability of $G$ and the Minkowski difference $K-B$, we obtain
\begin{align}
	\int_{K}\left|u(\br)\right|\dd\br
		&\leq	\int_{K}\int_{\R^d}\left|g(\by)G(\br-\by)\right|\dd\by\dd\br \notag \\
		&=		\int_{\R^d}\left|g(\by)\right|\int_{K}\left|G(\br-\by)\right|\dd\br\dd\by \notag \\
		&=		\int_{\R^d}\left|g(\by)\right|\int_{K-\by}\left|G(\bu)\right|\dd\bu\dd\by \notag \\
		&\leq	\int_{K-B}\left|G(\bu)\right|\dd\bu\,\int_{B}\left|g(\by)\right|\dd\by \notag \\
		&\leq	\norm{G}_{L^1(K-B)} \norm{g}_{L^1(\R^d)}. \label{eq:uL1loc}
\end{align}
Therefore, $u$ is locally integrable. By \eqref{eq:growth-G}, there exists a radius $R>0$ and a constant $C>0$ such that
\begin{align} \label{eq:uS'}
	\left|u(\br)\right| \leq \int_{B}\left|g (\by)G(\br-\by)\right|\dd\by\leq C \norm{g}_{L^1(\R^d)}
\end{align}
when $\left|\br\right|>R$. So, $u$ can be identified with a tempered distribution.

Concerning the Sommerfeld radiation condition let $\br \notin \supp g$. Then we have
\begin{equation*}
	\abs{\frac{\partial u}{\partial r}(\br) - \i k_0 u(\br)} \le \norm{g}_{L^1} \sup_{\by \in \supp g} \abs{\frac{\partial G}{\partial r}(\br - \by) - \i k_0 G(\br - \by)}.
\end{equation*}
Since $\supp g$ is compact, the right-hand side has the same asymptotic behavior for $\abs{\br} \to \infty$ as $\partial G/\partial r(\br) - \i k_0G(\br)$. Therefore, $u$ satisfies \eqref{eq:src}.

The uniqueness of $u$ follows from \autoref{thm:uniqueness}.
\end{proof}

\begin{remark} \label{rem:Lp-estimates}
  While \autoref{thm:src} only asserts that $u=g*G$ is a distributional solution of \eqref{eq:HH}, under slightly stronger assumptions it can be shown that $u$ is actually a strong solution. Specifically, let $d \ge 2$ and $g \in L^r(\R^d)$ with $r>\max(1,2d/(3+d))$ and compact support. Furthermore, let $p,q\ge1$ satisfy
    \begin{equation*}\label{eq:pq}
      p\le r,\quad 
      \frac{d+1}{2d}<\frac{1}{p}, \quad \frac{1}{q}<\frac{d-1}{2d}, \quad \frac{2}{d+1}\leq \frac{1}{p}-\frac{1}{q}\leq\frac{2}{d},  \quad \frac{1}{p}-\frac{1}{q}<1. 
    \end{equation*}
  Then one can apply the estimate
  \begin{equation*}
    \norm{\phi*G}_{L^q(\R^d)} \leq C \norm{\phi}_{L^p(\R^d)}, \quad \text{for all} \; \phi \in \mathcal{S}(\R^d),
  \end{equation*}
  from \cite{KenRuiSog87} (see also \cite[Thm.\ 2.1]{Eve17} and \cite[Thm.\ 6]{Gut04})
  to conclude that $u \in L^q(\R^d).$ Using elliptic regularity theory, one can argue that $u \in W^{2,p}_{\mathrm{loc}}(\R^d)$, see \cite[Prop.\ A.1]{EveWet15}.
\end{remark}

The proof of \autoref{thm:fourierdiffraction} requires the following continuity result for the map $g \mapsto u$. It takes into account the compact support of $g$ but otherwise requires less regularity than \autoref{rem:Lp-estimates}. We denote by $\B^d_R$ the open ball in $\R^d$ centered at $0$ with radius $R$.

\begin{theorem}\label{thm:Rcontinuous}
If $g_n \to 0$ in $L^1(\R^d)$ and $\bigcup_n \supp g_n$ is bounded, then $g_n*G \to 0$ in $\mathcal{S}'(\R^d)$.
\end{theorem}
\begin{proof}
Let $B = \bigcup_n \supp g_n$. There exists a sufficiently large $R>0$ such that for any $\phi\in \mathcal{S}(\R^d)$ we obtain
\begin{align*}
	\left|\int_{\R^d}g_n*G(\br)\phi(\br)\dd\br\right|
		&\leq\int_{\B^d_R}\left|g_n*G(\br)\phi(\br)\right|\dd\br+\int_{\R^d\backslash\B^d_R}\left|g_n*G(\br)\phi(\br)\right|\dd\br\\
		&\leq \left(\norm{G}_{L^1(\B^d_R-B)}\lVert\phi\rVert_{L^\infty(\R^d)}+C\lVert\phi\rVert_{L^1(\R^d)}\right)\lVert g_n\rVert_{L^1(\R^d)},
\end{align*}
where we have used \eqref{eq:uL1loc} for the first and \eqref{eq:uS'} for the second integral. For $n\to\infty$, we conclude that
\begin{equation*}
\int_{\R^d}g_n*G(\br)\phi(\br)\dd\br\rightarrow 0
\end{equation*}
and therefore $G*g_n\rightarrow 0$ in $\mathcal{S}^\prime$.
\end{proof}

\section{A generalized Fourier diffraction theorem} \label{sec:back}

We denote by $\ktran_{j}$ the partial Fourier transform with respect to the $j$-th coordinate. That is, if $\phi$ belongs to the Schwartz space $\mathcal{S}(\R^d)$, we define
\begin{equation*}
\mathcal{F}_{j} \phi(r_1,\ldots, r_{j-1}, k_j, r_{j+1},\ldots r_d) 
=  (2\pi)^{-\frac{1}{2}} \int_{\R} \phi(\br) \e^{-\i k_j r_j} \dd r_j.
\end{equation*}
Note that $\ktran_j$ can be extended to a continuous linear bijection with continuous inverse on the space of tempered distributions $\mathcal{S}'(\R^d)$. The usual $d$-dimensional Fourier transform is given by $\ktran = \ktran_{1} \circ \cdots \circ \ktran_{d}$. We also use the shorthand $\hat{g}$ for $\ktran g$.

If $d\geq 2$ the Fourier transform with respect to the first $d-1$ coordinates is abbreviated by
\begin{equation} \label{eq:Ftilde}
\tilde \ktran \coloneqq \ktran_{1} \circ \cdots \circ \ktran_{d-1}.
\end{equation}
Similarly, for $\br \in \R^d$ we define the truncated vector $\tilde \br = (r_1,\ldots, r_{d-1}) \in \R^{d-1}$. We also let
\begin{equation} \label{eq:kappa}
	\kappa: \R^{d-1} \to \C, \, \kappa (\tilde \br) \coloneqq
	\begin{cases}
		\sqrt{k_0^2-|\tilde \br|^2}, & |\tilde \br| \leq k_0,\\
		\i\sqrt{|\tilde \br|^2-k_0^2}, & |\tilde \br| > k_0.
	\end{cases}
\end{equation}
\begin{lemma}\label{thm:kappa} Let $d\ge 2$. Then $1/\kappa \in L^1_{\mathrm{loc}}(\R^{d})$ and $\phi \mapsto \int_{\R^d} \phi(\br)/\kappa(\tilde \br) \dd \br $ is a tempered distribution.
\end{lemma}
\begin{proof}
Concerning the first claim we have $\abs{\kappa(\tilde \br)} \ge \sqrt{k_0 \abs{k_0 - \abs{\tilde \br}}}$ and therefore, for every $R>0$,
\begin{equation} \label{eq:kappa_L1loc}
	\int_{\mathcal{B}^{d-1}_R} \frac{\dd \tilde \br}{\abs{\kappa(\tilde \br)}} \le \int_{\mathcal{B}^{d-1}_R} \frac{\dd \tilde \br}{\sqrt{k_0 \abs{k_0 - \abs{\tilde \br}}}} \le C \int_0^R \frac{r^{d-2}\dd r}{\sqrt{\abs{k_0 - r}}} < +\infty.
\end{equation}
This shows that $1/\kappa \in L^1_{\mathrm{loc}}(\R^{d-1})$ and consequently $1/\kappa \in L^1_{\mathrm{loc}}(\R^{d})$ as well.

For the second claim let $R>k_0$. Then $|\kappa| \ge \sqrt{k_0 (R-k_0)}$ on $S \coloneqq \{ \br \in \R^d : \abs{\tilde \br} \ge R \}$, and for an arbitrary $\phi \in \mathcal{S}(\R^d)$ it follows that
\begin{align*}
	\abs{\int_{\R^d} \frac{\phi(\br)}{\kappa(\tilde \br)} \dd \br}
		&=	\abs{\int_{\R} \int_{\mathcal{B}^{d-1}_R}   \frac{\phi(\br)}{\kappa(\tilde \br)} \dd \tilde \br \dd x_d
			+ \int_{S}  \frac{\phi(\br)}{\kappa(\tilde \br)} \dd \br} \\
		&\le \norm{1/\kappa}_{L^1(\mathcal{B}^{d-1}_R)} \int_{\R}  \sup_{|\tilde \br| \le R} \abs{\phi(\tilde \br, x_d)} \dd x_d
			+ (k_0 (R-k_0))^{-\frac12} \int_{S} \abs{\phi(\br)} \dd \br \\
		&\le C \left( \int_{\R}  \sup_{|\tilde \br| \le R} \abs{\phi(\br)} \dd x_d
			+ \int_{\R^d} \abs{\phi(\br)} \dd \br \right).
\end{align*}
Both of these integrals can be bounded by appropriate seminorms on $\mathcal{S}(\R^d)$. For the second one we recall that $\mathcal{S}(\R^d)$ is continuously embedded in $L^1(\R^d)$. The first one can be estimated by
\begin{align*}
	\int_{\R}  \sup_{|\tilde \br| \le R} \abs{\phi(\br)} \dd r_d
		&= \int_{\R}  \frac{1+r_d^2}{1+r_d^2}\sup_{|\tilde \br| \le R} \abs{\phi(\br)} \dd r_d
			\le \sup_{r_d \in \R} \left( (1+r_d^2)\sup_{|\tilde \br| \le R} \abs{\phi(\br)} \right) \int_{\R} \frac{\dd y_d}{1+y_d^2}\\
		&\le C \sup_{\br \in \R^d} \abs{ (1+r_d^2) \phi(\br)}.
\end{align*}
\end{proof}
\begin{lemma} \label{thm:G1}
Let $d \geq 2$. The partial Fourier transform $\tilde \ktran G$ is given by the locally integrable function
\begin{equation}\label{eq:G1}
	\tilde \ktran G(\br) = (2\pi)^{\frac{1-d}{2}} \frac{\i \e^{\i \kappa(\tilde \br) |r_d|}}{2 \kappa(\tilde \br)}.
\end{equation}
\end{lemma}
\begin{proof}
The $d$-dimensional Fourier transform $\hat G$ is a tempered distribution and can be expressed as 
\begin{equation} \label{eq:hat-G}
	\langle \hat G, \phi \rangle = (2\pi)^{-\frac{d}{2}} \lim_{\epsilon \to 0^+} \int_{\R^d} \frac{\phi(\bk)}{|\bk|^2 - k_0^2 -\i\epsilon} \dd \bk, \quad \phi \in \mathcal{S}(\R^d).
\end{equation}
This formula can be derived as follows. Given $\epsilon>0$, consider the Helmholtz operator $-\Delta-(k_0^2+\i\epsilon)$. An outgoing fundamental solution of this operator is given by
\begin{equation*}
G_\epsilon(\br) = \frac{\i}{4} \left( \frac{k}{2 \pi |\br|} \right)^\frac{d-2}{2} H_{\frac{d-2}{2}}^{(1)} (k\abs{\br})
\end{equation*}
where $k$ denotes the principal square root of $k_0^2+\i\epsilon$, see \cite[Chap.\ 9]{McL00}. Since $G_\epsilon$ is a fundamental solution of $-\Delta-(k_0^2+\i\epsilon)$ and a tempered distribution, its Fourier transform may be identified with the locally integrable function $\hat G_\epsilon(\by)=(2\pi)^{-d/2}(|\by|^2-(k_0^2+\i\epsilon))^{-1}$. Further, by using the asymptotic estimates \eqref{eq:Hankel-limit-small}-\eqref{eq:Hankel-limit-large}, see also \cite[(13)]{EveWet15}, and Lebesgue's dominated convergence theorem it can be shown that $G_\epsilon\overset{\mathcal{S}^\prime}{\rightarrow} G$ and therefore also $\hat{G_\epsilon}\overset{\mathcal{S}^\prime}{\rightarrow}\hat{G}$, which is \eqref{eq:hat-G}.

Exploiting the fact that $\tilde \ktran = \ktran_d^{-1} \ktran$ we obtain
\begin{align*}
	\langle \tilde \ktran G, \phi \rangle
		&= (2\pi)^{-\frac{d}{2}} \lim_{\epsilon \to 0^+} \int_{\R^d} \frac{ \ktran_d^{-1} \phi(\bk)}{|\bk|^2 - k_0^2 -\i \epsilon} \dd \bk 
		= (2\pi)^{-\frac{d+1}{2}} \lim_{\epsilon \to 0^+} \int_{\R^d} \int_\R \frac{ \e^{\i r_d k_d} \phi(\tilde \bk,r_d)}{|\bk|^2 - k_0^2 -\i \epsilon} \dd r_d \dd \bk. 
\end{align*}
Applying Fubini's theorem to interchange integration with respect to $k_d$ and $r_d$ we obtain
\begin{equation*}
	\langle \tilde \ktran G, \phi \rangle = (2\pi)^{-\frac{d+1}{2}} \lim_{\epsilon \to 0^+} \int_{\R^d} \phi(\tilde \bk,r_d) \int_\R \frac{ \e^{\i r_d k_d}}{|\bk|^2 - k_0^2 -\i \epsilon} \dd k_d \dd(\tilde \bk, r_d).
\end{equation*}
Concerning the inner integral define $\kappa_\epsilon$ as the square root of $\kappa^2 + \i \epsilon$ with positive imaginary part and use formula 17.23.14 in \cite{GrRy07} to obtain
\begin{equation*}
	\int_\R \frac{ \e^{\i r_d k_d}}{|\bk|^2 - k_0^2 -\i \epsilon} \dd k_d = \int_\R \frac{ \e^{\i r_d k_d}}{k_d^2 - \kappa(\tilde \bk)^2 -\i \epsilon} \dd k_d = \int_\R \frac{ \e^{\i r_d k_d}}{k_d^2 + (-\i\kappa_\epsilon(\tilde \bk))^2} \dd k_d = \frac{\pi \e^{\i \kappa_\epsilon(\tilde \bk) |r_d|} }{-\i \kappa_\epsilon(\tilde \bk)}.
\end{equation*}
Thus, we have shown that
\begin{align}\label{eq:domcon}
	\langle \tilde \ktran G,\phi\rangle = (2\pi)^{\frac{1-d}{2}} \lim_{\epsilon \to 0^+} \int_{\R^d} \phi(\tilde \bk, r_d) \frac{\i \e^{\i \kappa_\epsilon(\tilde \bk) |r_d|} }{2 \kappa_\epsilon(\tilde \bk)} \dd(\tilde \bk, r_d).
\end{align}
Finally, we can interchange limit and integral by Lebesgue's dominated convergence theorem, since $|\kappa_\epsilon| > |\kappa|$ and the integrand in \eqref{eq:domcon} is dominated by $|\phi|/|2\kappa|$, which is in $L^1(\R^d)$ according to \autoref{thm:kappa}. Therefore,
\begin{align*}
	\langle \tilde \ktran G,\phi\rangle = (2\pi)^{\frac{1-d}{2}} \int_{\R^d} \phi(\tilde \bk, r_d) \frac{\i \e^{\i \kappa(\tilde \bk) |r_d|} }{2 \kappa(\tilde \bk)} \dd(\tilde \bk, r_d).
\end{align*}
\end{proof}
Having found $\tilde \ktran G$ we calculate $\tilde \ktran u$ in \autoref{thm:fourierdiffraction}. For $d \ge 2$ define
\begin{equation} \label{eq:h}
	\bh^\pm \colon \R^{d-1} \to \C^d, \, \bh^\pm\left(\tilde \br \right) \coloneqq \begin{pmatrix} \tilde \br \\ \pm \kappa(\tilde \br) \end{pmatrix}.
\end{equation}
We also introduce the half space
\begin{align*}
	E_a \coloneqq \{\by \in \R^d : y_d \ge a\}, \quad a \in \R,
\end{align*}
and the indicator function $\mathbf{1}_A$ of a set $A$, which is defined by $\mathbf{1}_A(x) \coloneqq 1$ if $x\in A$ and $\mathbf{1}_A(x) \coloneqq 0$ otherwise.
Regarding the right-hand side of \eqref{eq:F12u} below, we recall that the Fourier transform of a function with compact support in $\R^d$ can be extended to an entire function on $\C^d.$

\begin{theorem}[Generalized Fourier Diffraction Theorem]\label{thm:fourierdiffraction}
Let $d\ge 2$. Assume that $g \in L^1(\R^d)$ has compact support.
Then $\tilde \ktran u$, where  $u = g*G$, is given by the following locally integrable function
\begin{equation} \label{eq:F12u}
	\tilde \ktran u(\br) = \sqrt{\frac{\pi}{2}} \frac{\i}{\kappa(\tilde \br)} \left(\e^{\i \kappa(\tilde \br) r_d} \ktran \big((1-\mathbf{1}_{E_{r_d}})g \big)(\bh^+(\tilde \br)) + \e^{-\i \kappa(\tilde \br) r_d} \ktran \left(\mathbf{1}_{E_{r_d}}g \right)(\bh^-(\tilde \br)) \right),
\end{equation}
for $\br\in\R^d$ with $\abs{\tilde\br} \neq k_0$.
If $r_d$ is sufficiently large or sufficiently small such that
\begin{equation}\label{eq:largexd}
  \pm (r_d - y_d) > 0\quad \text{for all } \by\in \supp g,
\end{equation}
then \eqref{eq:F12u} simplifies to
\begin{equation}\label{eq:F12u-outisde}
  \tilde \ktran u (\br) = \sqrt{\frac\pi2}\, \frac{ \i \e^{\pm\i \kappa(\tilde \br) r_d}}{\kappa(\tilde \br)} \hat{g}(\bh^\pm (\tilde \br)).
\end{equation}
\end{theorem}
\begin{proof}
There is a compact set $K\subset\R^d$ and a sequence $(g_n) \subset \mathcal{D}(\R^d)$ converging to $g$ in $L^1(\R^d)$ such that $K$ contains the supports of $g$ and all $g_n$. Define $u_n  = g_n * G.$ Then, in the sense of tempered distributions we have
\begin{align*}
	\tilde \ktran u_n = \ktran_d^{-1} \ktran (g_n*G) = (2\pi)^{\frac{d}{2}} \ktran_d^{-1} (\hat g_n \hat G) = (2\pi)^{\frac{d-1}{2}} \tilde \ktran g_n \stackrel{d}{*} \tilde \ktran G,
\end{align*}
the last equality being a consequence of the convolution theorem for partial Fourier transforms and $\stackrel{d}{*}$ denoting partial convolution along the $d$-th coordinate, see \cite[Def.~8.21, Thm.~8.22]{KirQueRitSchSet21}. Taking into account \autoref{thm:G1}, it follows that
\begin{equation*}
	\tilde \ktran u_n(\br)=\int_{\R} \frac{\i \e^{\i \kappa(\tilde \br) |r_d-y|}}{2 \kappa(\tilde \br)} \tilde \ktran g_n(\tilde \br, y) \dd y
\end{equation*}
wherever $\kappa \neq 0$. So, for every $\phi \in \mathcal{S}(\R^d)$
\begin{gather}
	\langle \tilde \ktran u_n, \phi \rangle
		= \frac{\i}{2}\int_{\R^d} \frac{\phi (\br)}{\kappa(\tilde \br)}  \int_{\R} \e^{\i \kappa(\tilde \br) |r_d-y|} \tilde \ktran g_n(\tilde \br, y) \dd y \dd \br.
		\label{eq:Fu_n1}
\end{gather}
Next, we rewrite the inner integral as a sum of two $d$-dimensional Fourier transforms.
\begin{align*}
	&\int_{\R} \e^{\i \kappa(\tilde \br) |r_d-y|} \tilde \ktran g_n(\tilde \br, y) \dd y \\
		&\qquad = \int_{-\infty}^{r_d} \e^{\i \kappa(\tilde \br) (r_d-y)} \tilde \ktran g_n(\tilde \br, y) \dd y
			+ \int^{+\infty}_{r_d} \e^{\i \kappa(\tilde \br) (y-r_d)} \tilde \ktran g_n(\tilde \br, y) \dd y \\
		&\qquad = \e^{\i \kappa(\tilde \br)r_d} \int_{\R} \mathbf{1}_{(-\infty,r_d)}(y) \e^{-\i \kappa(\tilde \br)y} \tilde \ktran g_n(\tilde \br, y) \dd y
			+ \e^{-\i \kappa(\tilde \br)r_d} \int_{\R} \mathbf{1}_{(r_d,\infty)}(y) \e^{\i \kappa(\tilde \br)y} \tilde \ktran g_n(\tilde \br, y) \dd y \\
		&\qquad = \frac{\e^{\i \kappa(\tilde \br)r_d}}{(2\pi)^{\frac{d-1}{2}}} \int_{\R^d} \mathbf{1}_{(-\infty,r_d)}(y) g_n(\tilde \bx, y) \e^{-\i (\tilde \bx \cdot \tilde \br + \kappa(\tilde \br)y)}  \dd (\tilde \bx, y) \\
		&\qquad \quad + \frac{\e^{-\i \kappa(\tilde \br)r_d}}{(2\pi)^{\frac{d-1}{2}}} \int_{\R^d} \mathbf{1}_{(r_d,\infty)}(y) g_n(\tilde \bx, y) \e^{-\i (\tilde \bx \cdot \tilde \br - \kappa(\tilde \br)y)}  \dd (\tilde \bx, y) \\
		&\qquad = (2\pi)^{\frac12} \left( \e^{\i \kappa(\tilde \br)r_d} \ktran \big((1-\mathbf{1}_{E_{r_d}})g_n\big)(\bh^+(\tilde \br)) + \e^{-\i \kappa(\tilde \br)r_d} \ktran \big(\mathbf{1}_{E_{r_d}}g_n\big)(\bh^-(\tilde \br)) \right)
\end{align*}
Therefore \eqref{eq:Fu_n1} equals
\begin{gather}
	\langle \tilde \ktran u_n, \phi \rangle
		= \i \sqrt{\frac{\pi}{2}} \int_{\R^d} \frac{\phi(\br)}{\kappa(\tilde \br)} \left(\e^{\i \kappa(\tilde \br) r_d} \ktran \big((1-\mathbf{1}_{E_{r_d}})g_n \big)(\bh^+(\tilde \br)) 
		+ \e^{-\i \kappa(\tilde \br) r_d} \ktran \big(\mathbf{1}_{E_{r_d}}g_n \big)(\bh^-(\tilde \br)) \right) \dd \br.
		\label{eq:Fu_n2}
\end{gather}
Now consider the limit $n \to \infty$ in \eqref{eq:Fu_n2}. Regarding the left-hand side, we have $u_n\to u$ in $\mathcal{S}^\prime(\R^d)$ by \autoref{thm:Rcontinuous}, since  $\bigcup_n \supp g_n\subset K$ is bounded. Continuity of $\tilde \ktran$ on $\mathcal{S}^\prime(\R^d)$ gives
\begin{equation*}
	\lim_{n\to\infty}\langle \tilde \ktran u_n, \phi \rangle = \langle \tilde \ktran u, \phi \rangle,\quad \text{for all } \phi\in \mathcal{S}(\R^d).
\end{equation*}
To resolve the limit on the right-hand side we use the dominated convergence theorem. The pointwise limit of the integrand is given by
\begin{equation}\label{eq:pointwise}
	\frac{\phi(\br)}{\kappa(\tilde \br)} \left(\e^{\i \kappa(\tilde \br) r_d} \ktran \big((1-\mathbf{1}_{E_{r_d}})g \big)(\bh^+(\tilde \br)) 
		+ \e^{-\i \kappa(\tilde \br) r_d} \ktran \big(\mathbf{1}_{E_{r_d}}g \big)(\bh^-(\tilde \br)) \right).
\end{equation}
To see this, we note that $g_n \to g$ in $L^1(K)$ and also $\mathbf{1}_{E_{r_d}}g_n \to \mathbf{1}_{E_{r_d}}g$ in $L^1(K)$ for all $r_d \in \R$. Consequently, $\ktran (\mathbf{1}_{E_{r_d}}g_n) \to \ktran (\mathbf{1}_{E_{r_d}}g)$ pointwise on $\C^d$, cf.\ \cite[(7.3.1)]{Ho90}. Hence $\ktran (\mathbf{1}_{E_{r_d}}g_n)(\bh^-(\tilde \br)) \to \ktran (\mathbf{1}_{E_{r_d}}g)(\bh^-(\tilde \br))$ for all $\br \in \R^d.$ Analogously, we find that $\ktran ((1-\mathbf{1}_{E_{r_d}})g_n)(\bh^+(\tilde \br)) \to \ktran ((1-\mathbf{1}_{E_{r_d}})g)(\bh^+(\tilde \br))$ for all $\br \in \R^d.$

Next, it follows from \eqref{eq:Fu_n1} that the integrand is bounded by $\sqrt{2\pi} \norm{g_n}_{L^1} \abs{\phi/2\kappa}$. Since we can find a $C$ such that $\norm{g_n}_{L^1} \le C \norm{g}_{L^1}$ for all $n$ we have found an upper bound.

Finally, it follows from the dominated convergence theorem that \eqref{eq:pointwise} is in $L^1(\R^d)$ and
\begin{align*}
	\langle \tilde \ktran u, \phi \rangle
		= \i \sqrt{\frac{\pi}{2}} \int_{\R^d} \frac{\phi(\br)}{\kappa(\tilde \br)} \left(\e^{\i \kappa(\tilde \br) r_d} \ktran \big((1-\mathbf{1}_{E_{r_d}})g \big)(\bh^+(\tilde \br)) + \e^{-\i \kappa(\tilde \br) r_d} \ktran \big(\mathbf{1}_{E_{r_d}}g \big)(\bh^-(\tilde \br)) \right) \dd \br
\end{align*}
for all $\phi \in \mathcal{S}(\R^d)$, which proves \eqref{eq:F12u}. Furthermore, if \eqref{eq:largexd} is fulfilled, one of the two Fourier transforms on the right-hand side of \eqref{eq:F12u} vanishes, while the other one equals $\hat g$, so that we obtain \eqref{eq:F12u-outisde}, which finishes the proof.
\end{proof}

\begin{remark}[1D Fourier diffraction theorem]
\autoref{thm:fourierdiffraction} can be extended to dimension $d=1$ in the following way. Let $g\in L^1(\R)$ be compactly supported and recall from \eqref{eq:G123} the simple expression of the one-dimensional fundamental solution. Then a direct calculation yields
\begin{align*}
	u(x) &= \frac{\i}{2k_0} \int_\R g(y) \e^{\i k_0 \abs{x-y}} \dd y 
		= \sqrt{\frac{\pi}{2}} \frac{\i}{k_0} \left(\e^{\i k_0 x} \ktran \left((1-\mathbf{1}_{E_{x}})g \right)(k_0) + \e^{-\i k_0 x} \ktran \left(\mathbf{1}_{E_{x}}g \right)(-k_0) \right).
\end{align*}
If $x$ lies outside the support of $g$, that is, $\pm (x-y) > 0$ for all $y\in \supp g$, then
\begin{equation*}
	u(x) =  \sqrt{\frac{\pi}{2}} \frac{\i}{k_0} \e^{\pm \i k_0 x} \hat g(\pm k_0).
\end{equation*}
\end{remark}

\begin{remark}
Since $g$ has compact support, $\tilde \ktran u$ is smooth wherever $|\tilde \br| \neq k_0$ according to \eqref{eq:F12u}. 
On the other hand, consider $\tilde  \br_0 \in \R^{d-1}$ with $\abs{\tilde \br_0} = k_0$ and note that $\ktran ((1-\mathbf{1}_{E_{r_d}})g)\circ \bh^+$ and $\ktran (\mathbf{1}_{E_{r_d}}g)\circ \bh^-$ are continuous functions on $\R^{d-1}$. Evaluating the long bracket in \eqref{eq:F12u} at $(\tilde  \br_0,r_d)$ gives $\hat g (\tilde \br_0,0)$ for every $r_d\in \R.$
Thus, the function $\tilde \ktran u$ has a singularity at $(\tilde \br_0,r_d)$, for every $r_d \in \R$, if $\hat g (\tilde \br_0,0) \neq 0$.
\end{remark}

\section{Fourier coverage} \label{sec:coverage}
In this section we investigate some of the ramifications of \autoref{thm:fourierdiffraction} for data collection strategies in diffraction tomography. Therefore, we return to the inverse scattering problem outlined in \autoref{sec:exp}. Under the Born or Rytov approximation, cf.\ \eqref{eq:hhborn} and \eqref{eq:hhrytov}, the governing equation is
\begin{equation}\label{eq:bornrytov}
	-(\Delta + k_0^2)u(\br) = k_0^2 f(\br)\ui(\br), \quad \br \in \R^d,
\end{equation}
where $\ui$ is the incident wave and the outgoing solution $u$ approximates the scattered wave. The normalized scattering potential $f$, recall \eqref{eq:f}, is the unknown we aim to reconstruct. From now on we impose the following assumptions, which are standard in diffraction tomography.
\begin{enumerate}
	\item \label{item:planewave} The incident field is a plane wave $\ui(\br) = \e^{\i k_0 \bs \cdot \br}$ for some $\bs \in \mathbb{S}^{d-1}$.
	\item \label{item:outside} The measurement hyperplane $\{\br \in \R^d : r_d = \rM \}$ is disjoint from $\supp f$, i.e.\ condition \eqref{eq:largexd} holds. Introducing the intervals
		\begin{align}\label{eq:I(f)}
			I^\pm = I^\pm(f) = \{x \in \R : x \gtrless y_d \text{ for all } \by \in \supp f \},
		\end{align}
		this condition can be written as $\rM \in I^\pm$.
\end{enumerate}
\begin{figure}
  \begin{center}
    \begin{tikzpicture}[scale=0.43,
      >=stealth',
      pos=.8,
      photon/.style={decorate,decoration={snake,post length=1mm}}
      ]
      
      \draw[->] (2,0)--(11,0)  ;
      \draw[->] (2,0)--(2,5)  ;
      \draw[->] (2,0)--(-0.5,-2.5)  ;
      \node at (11.6,0) {$r_3$};
      \node at (2.8,4.8) {$r_1$};
      \node at (-1,-3) {$r_2$};
      
      \fill[blue,opacity=0.3] (2,0) circle (2.4);
      \node[blue] at (3.5,-2.9) {$\supp f$};
      
      \draw[->,red] (4.5,0)--(11,0)  ;
      \node[red] at (6.5,0.55) {$I^+(f)$};
      \node[red] at (4.5,0) {$\big($};
      
      \draw[dashed] (-2.875,-3) -- (-2,3.625);
      \draw[dashed]  (-3.875,-2.875) -- (-3,3.75);
      \draw[dashed]  (-4.875,-2.75) -- (-4,3.875);
      \draw[dashed] (-5.875,-2.625) -- (-5,4);
      \draw[->] (-8,1) -- (-6,0.75);
      \node at (-6.75,1.5) {$\bs$};	
      \node at (-2.5,4.5) {$\ui$};	
      \draw[decorate,decoration={brace,amplitude=4pt,mirror,raise=0ex}] (-5.875,-2.625) -- (-4.875,-2.75) node[midway,yshift=-1em]{$2\pi/k_0$};
      
      \draw[line width = 3pt] (9,-4) -- (9,4);
      \node at (9,4.5) {$r_3 = r_M$};
    \end{tikzpicture}
  \end{center}
  \caption{Measurement setup for $d=3$ with the open interval $I^+(f)$ marked in red along the $r_3$ axis. The incident field has direction $\bs$ and wavelength $2\pi/k_0$.}
  \label{fig:setup}
\end{figure}
\autoref{fig:setup} illustrates the measurement setup.
The following $d$-dimensional version of the Fourier diffraction theorem, see also \cite{KakSla01,NatWue01,Wol69}, is now an immediate consequence of \autoref{thm:fourierdiffraction}.
Recall the partial Fourier transform $\tilde\ktran$ in \eqref{eq:Ftilde}, $ \kappa$ in \eqref{eq:kappa} and $\bh$ in \eqref{eq:h}.
\begin{corollary}[Fourier diffraction theorem]\label{thm:fdt}
Let $d\ge 2$ and assume that $f \in L^1(\R^d)$ has compact support. Then, for $\bx \in \R^{d-1}$ with $\abs{\bx} \neq k_0$, the outgoing solution $u$ of \eqref{eq:bornrytov} satisfies
\begin{equation} \label{eq:recon}
  \tilde \ktran u(\bx,\rM) =
  \sqrt{\frac\pi2}\, \frac{ \i \e^{\pm\i \kappa(\bx) \rM} k_0^2}{\kappa(\bx)} \hat f(\bh^{\pm}(\bx) - k_0 \bs),\qquad  \text{if } \rM \in I^\pm(f).
\end{equation}
\end{corollary}

\begin{remark}[Variants of the Fourier diffraction theorem] \label{rem:applicability}
In the present article, the main application of \autoref{thm:fourierdiffraction} is the special case \autoref{thm:fdt} based on assumptions \ref{item:planewave} and \ref{item:outside}. On the other hand, it is precisely the absence of these assumptions which makes \autoref{thm:fourierdiffraction} more general and potentially more widely applicable. That is, \autoref{thm:fourierdiffraction} could be used in situations, where the data are collected on a hyperplane passing through the inhomogeneity or where the incident field $\ui$ is not a plane wave. Consider, for instance, an incident Herglotz wave
\begin{equation*}
	\ui(\br) = \int_{\mathbb{S}^{d-1}} a(\bs) \e^{\i k_0\bs\cdot \br} \dd s(\bs),
\end{equation*}
where $a \in L^2(\mathbb{S}^{d-1})$. Replacing $g$ in \eqref{eq:F12u-outisde} with $k_0^2 f \ui$ and changing the order of integration yields
\begin{equation*}
  \tilde \ktran u(\bx,\rM) =
  \sqrt{\frac\pi2}\, \frac{ \i \e^{\pm\i \kappa(\bx) \rM} k_0^2}{\kappa(\bx)} \int_{\mathbb{S}^{d-1}}  a(\bs) \hat f(\bh^{\pm}(\bx) - k_0 \bs) \dd s(\bs).
\end{equation*}
Such relations between $\tilde \ktran u$ and $\hat f$ have recently been used for tomographic reconstructions in \cite{KirNauSchYan24}. Moreover, we note that the applicability of \autoref{thm:fourierdiffraction} is not restricted to outgoing solutions: Suppose $w$ is an arbitrary solution of \eqref{eq:HH}. Then $w = u+v$, where $u$ is the outgoing solution and $\Delta v + k_0^2 v = 0$ on $\R^d$. If $\tilde \ktran v$ can be calculated, then $\tilde \ktran w = \tilde \ktran u + \tilde \ktran v$ and, using \autoref{thm:fourierdiffraction}, one obtains a formula for $\tilde \ktran w$.
Finally, we remark that there are vector-valued versions of the Fourier diffraction theorem, see \cite{Lau02,MejSch24}.
\end{remark}

So far all parameters of the experiment were kept fixed.
In that case the Fourier diffraction theorem in \eqref{eq:recon} gives information about $\mathcal{F}f$ on the hemisphere 
\begin{equation} \label{eq:hemisphere}
	\{\bh^\pm(\bx) - k_0\bs : \bx\in\R^{d-1},\, \abs{\bx}<k_0\}
\end{equation}
with center $-k_0\bs$ and radius $k_0$.
This is called the \emph{Fourier coverage} or \emph{k-space coverage} of the experiment and we denote it by $\mathcal{Y}\subset\R^d$.
The restriction $\abs{\bx}<k_0$ is made for the practical reason that the larger spatial frequencies do not contribute to the physical measurements.

The set in \eqref{eq:hemisphere}, however, is only a null set.
For a viable reconstruction, we need to obtain more information, namely a larger Fourier coverage, by adapting the experimental setup.
In this section, we discuss how altering 
\begin{enumerate}\vspace{-\topsep-4pt} \setlength{\parskip}{0pt}
	\item the direction of incidence $\bs$,
	\item the orientation and position of the object,
	\item the orientation and position of the measurement hyperplane, or
	\item the wave number $k_0$
\end{enumerate}\vspace{-\topsep-4pt}
affects the Fourier coverage $\mathcal{Y}$. Regarding the first three constituents, the decisive factor is their orientation \emph{relative to each other}. A change in one of them is equivalent to a corresponding change in the other two. For instance, measurements obtained from rotating the object during illumination can be reproduced, at least theoretically, by rotating the direction of incidence and the measurement equipment in a corresponding fashion. Altering the wave number $k_0$ is different in character and will be treated in \autoref{sec:cover-frequencies}.
Finally, as a preparation for the general filtered backpropagation presented in \autoref{sec:backprop}, we consider in \autoref{sec:coverage-general} an experiment where all the above constituents may vary simultaneously.

\subsection{Direction of incidence}
\label{sec:incidence}
Altering the incidence direction $\bs$ is known as \emph{angle scanning} \cite{LeeShiYaqSoPar19} or illumination scanning \cite{ParShiPar18}. Then instead of \eqref{eq:recon}, we obtain for $\bx\in\R^{d-1}$
\begin{equation} \label{eq:recon_n2}
\tilde \ktran u_t(\bx,\rM)= \sqrt{\frac\pi2} \frac{ \i \e^{\pm \i \kappa \rM} k_0^2}{\kappa} \ktran f \left(\bh^\pm -k_0 \bs(t) \right), \qquad \text{if } \rM \in I^\pm,
\end{equation}
where $u_t$, $0\le t \le L$, is the scattered wave generated by the incident plane wave $\br\mapsto\e^{\i k_0 \br\cdot \bs(t)}$ and $\bs\colon[0,L] \to \mathbb{S}^{d-1}$ is the varying direction of incidence. Thus the Fourier coverage is given by
\begin{equation*}
	\mathcal{Y} = \left\{ \bh^\pm(\bx) -k_0 \bs(t) \in \R^d : |\bx| < k_0, \, 0\le t \le L \right\}.
\end{equation*}
Geometrically speaking, it consists of translations of the semicircle or hemisphere \eqref{eq:hemisphere} such that its center stays at a distance of $k_0$ from the origin, see \autoref{fig:rot_wave-2d}.
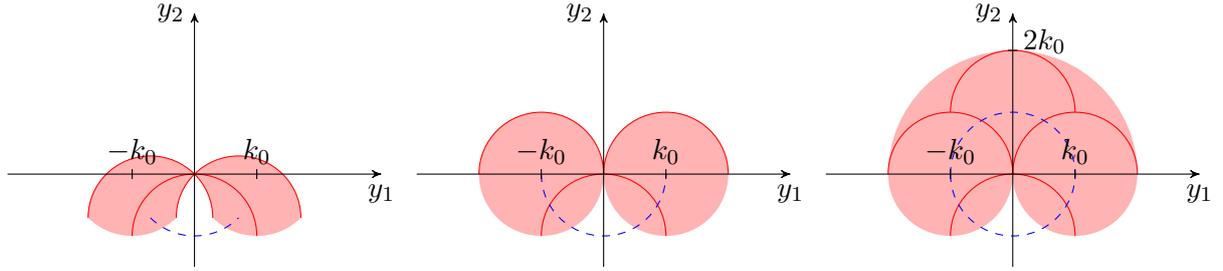
\begin{figure}
\centering

\begin{tikzpicture}[scale=.82] 
  \fill[red,opacity=.3] (1.707,-.707) arc (0:135:1cm) -- (0,0) arc(45:0:1cm); 
  \fill[red,opacity=.3] (0.293,-.707) arc (225:315:1cm) ; 
  \fill[red,opacity=.3] (-1.707,-.707) arc (180:45:1cm) -- (0,0) arc(135:180:1cm); 
  \fill[red,opacity=.3] (-0.293,-.707) arc (315:225:1cm); 
  
  \draw[red] (.293,-.707) arc (0:180:1cm);
  \draw[red] (1.707,-.707) arc (0:180:1cm);
  \draw[red] (1,-1) arc (0:180:1cm);
  \draw[blue,dashed] (-.707,-.707) arc (225:315:1cm);
  
  \draw[->] (0,-1.5) -- (0,2.6) node[anchor = east]{$y_2$};
  \draw[->] (-3,0) -- (3,0) node[anchor = north]{$y_1$};
  
  \draw (1,-0.08) -- (1,0.08);
  \node[anchor = south] at (1,0){$k_0$};
  \draw (-1,-0.08) -- (-1,0.08);
  \node[anchor = south] at (-1,0){$-k_0$};
\end{tikzpicture}\hfill
\begin{tikzpicture}[scale=.82]
    \fill[red,opacity=.3] (2,0) arc (0:360:1cm);
    \fill[red,opacity=.3] (0,0) arc (0:360:1cm);
    
    \draw[red] (0,0) arc (0:180:1cm);
    \draw[red] (2,0) arc (0:180:1cm);
    \draw[red] (1,-1) arc (0:180:1cm);
    \draw[blue,dashed] (-1,0) arc (180:360:1cm);

	\draw[->] (0,-1.5) -- (0,2.6) node[anchor = east]{$y_2$};
	\draw[->] (-3,0) -- (3,0) node[anchor = north]{$y_1$};

	\draw (1,-0.08) -- (1,0.08);
    \node[anchor = south] at (1,0){$k_0$};
    \draw (-1,-0.08) -- (-1,0.08);
    \node[anchor = south] at (-1,0){$-k_0$};
\end{tikzpicture}\hfill
\begin{tikzpicture}[scale=.82]
    \fill[red,opacity=.3] (2,0) arc (0:180:2cm);
    \fill[red,opacity=.3] (2,0) arc (0:-180:1cm);
    \fill[red,opacity=.3] (0,0) arc (0:-180:1cm);
    
    \draw[red] (0,0) arc (0:180:1cm);
    \draw[red] (2,0) arc (0:180:1cm);
    \draw[red] (1,-1) arc (0:180:1cm);
    \draw[red] (1,1) arc (0:180:1cm);
    \draw[blue,dashed] (-1,0) arc (-180:180:1cm);

	\draw[->] (0,-1.5) -- (0,2.6) node[anchor = east]{$y_2$};
	\draw[->] (-3,0) -- (3,0) node[anchor = north]{$y_1$};
	
	\draw (1,-0.08) -- (1,0.08);
    \node[anchor = south] at (1,0){$k_0$};
    \draw (-1,-0.08) -- (-1,0.08);
    \node[anchor = south] at (-1,0){$-k_0$};
	\draw (-0.08,2) -- (0.08,2);
    \node[anchor = west] at (0,2.15){$2k_0$};
\end{tikzpicture}

\caption{2D Fourier coverage for incidence direction varying according to $\bs(t) = (\cos t, \sin t)$ where $t \in [\pi/4, 3\pi/4]$ (left), $t \in [0,\pi]$ (center) and $t \in [0,2\pi]$ (right). Measurements are taken at $r_2=\rM$ with $\rM \in I^+$, recall \eqref{eq:I(f)}. The Fourier coverage (light red) is a union of infinitely many semicircles, some of which are depicted in red. Their centers lie on the dashed blue curve.}
\label{fig:rot_wave-2d}
\end{figure}

\subsection{Rigid motion of object} \label{sec:cover-rotation}
If the object moves according to a rigid motion $(t,\br) \mapsto R(t)^\top \br + \bd(t)$ with a rotation matrix
$$R(t) \in SO(d) \coloneqq \{Q\in\R^{d\times d} : Q^\top Q=I,\, \det Q=1\}$$ 
and a translation vector $\bd(t) \in \R^d$, $t\in[0,L]$,
it has the normalized scattering potential $f\circ\Psi_t$ with
\begin{equation}\label{eq:motion}
	\Psi_t\colon\R^d\to\R^d,\quad\br\mapsto R(t)(\br-\bd(t)).
\end{equation}
We denote by $u_t= k_0^2 ((f\circ\Psi_t) \ui)*G$ the wave scattered by this transformed potential and assume that $\rM \in I^\pm(f\circ\Psi_t)$ for all $t\in[0,L]$. Then \eqref{eq:recon} becomes
\begin{equation} \label{eq:recon_n}
	\tilde \ktran u_t(\bx,\rM)= \sqrt{\frac\pi2} \frac{ \i \e^{\pm\i \kappa \rM} k_0^2}{\kappa} \ktran f \left(R(t) \left(\bh^\pm -k_0 \bs \right) \right)
  \, \e^{-\i {\bd(t)}\cdot{(\bh^\pm-k_0 \bs)}},
\end{equation} 
cf.~\cite[sect.\ 2.2]{ElbQueSchSte23}. In this case we obtain the Fourier coverage
\begin{equation*}
	\mathcal{Y} = \left\{ R(t) \left( \bh^\pm(\bx) -k_0 \bs \right)\in\R^d : |\bx| < k_0, \, 0\le t \le L \right\},
\end{equation*}
which depends only on the rotation $R$ but not on the translation $\bd$.
It consists of rotated versions of the semicircle or hemisphere from \eqref{eq:hemisphere}, see \autoref{fig:2d-cover-half}, \autoref{fig:2d-max-cover} and \autoref{fig:cover3}. Comparing with \autoref{fig:rot_wave-2d} shows that rotating the object is not equivalent to rotating the incidence in terms of Fourier coverage. 

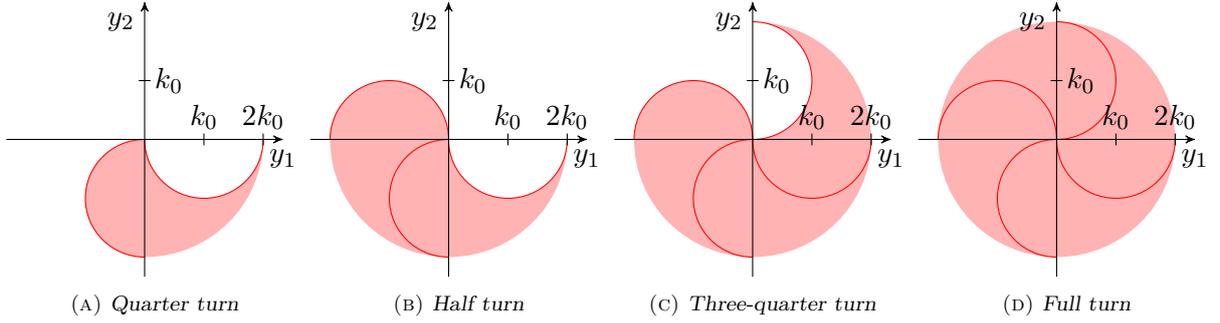
\begin{figure}[ht!]\centering
  \begin{subfigure}{.33\textwidth}\centering
    \begin{tikzpicture}[scale=.65]
      \fill[red,opacity=.3] (-1.5,-1.5) arc (-135:-45:2.12) -- (1.5,-1.5) arc (-90:-180:1.5) -- (0,0) arc (90:180:1.5);
      \fill[red,opacity=.3] (1.5,-1.5) arc (-45:45:2.12) -- (1.5,1.5) arc (90:180:1.5) -- (0,0) arc (90:0:1.5);
      
      \draw[red] (-1.5,-1.5) arc (180:0:1.5);
      \draw[red] (1.5,1.5) arc (90:270:1.5);
      
      \draw[->]   (-3,0) -- (3,0) node[below,black] {$y_1$};
      \draw[->]   (0,-2.5) -- (0,2.8) node[below left,black] {$y_2$};
      
      \draw (2.12,-0.15) -- (2.12,0.15);
      \node[above] at (1.65,0) {$\sqrt{2}k_0$};
      \draw (-0.15,1.5) -- (0.15,1.5);
      \node[right] at (0.1,1.5) {$k_0$};
      \draw (-0.15,-1.5) -- (0.15,-1.5);
      \node[right] at (0.1,-1.5) {$-k_0$};
    \end{tikzpicture}
    \caption{Quarter turn} \label{fig:2d-cover-half-1}
  \end{subfigure}\hfill
  \begin{subfigure}{.33\textwidth}\centering
	\begin{tikzpicture}[scale=.65]
    		\fill[red,opacity=.3] (-1.5,1.5) arc (135:-135:2.12) -- (-1.5,-1.5) arc (180:90:1.5) -- (0,0) arc (270:180:1.5);
    		
    		\draw[red] (-1.5,1.5) arc (180:360:1.5);
    		\draw[red] (-1.5,-1.5) arc (180:0:1.5);
    		\draw[red] (1.5,1.5) arc (90:270:1.5);
    		
    		\draw[->]   (-3,0) -- (3,0) node[below,black] {$y_1$};
      	\draw[->]   (0,-2.5) -- (0,2.8) node[below left,black] {$y_2$};
      	
      	\draw (2.12,-0.15) -- (2.12,0.15);
      	\node[above] at (1.65,0) {$\sqrt{2}k_0$};
      	\draw (-0.15,1.5) -- (0.15,1.5);
	    \node[right] at (0.1,1.5) {$k_0$};
	    \draw (-0.15,-1.5) -- (0.15,-1.5);
	    \node[right] at (0.1,-1.5) {$-k_0$};
    \end{tikzpicture}
  \caption{Half turn}
  \end{subfigure}\hfill
  \begin{subfigure}{.33\textwidth}\centering
    \begin{tikzpicture}[scale=0.65]
    		\fill[red,opacity=.3] (0,0) circle (2.12);
    		
    		\draw[red] (-1.5,1.5) arc (180:360:1.5);
    		\draw[red] (-1.5,-1.5) arc (180:0:1.5);
    		\draw[red] (1.5,1.5) arc (90:270:1.5);
    		\draw[red] (-1.5,1.5) arc (90:-90:1.5);
    		
		\draw[->]   (-3,0) -- (3,0) node[below,black] {$y_1$};
      	\draw[->]   (0,-2.5) -- (0,2.8) node[below left,black] {$y_2$};

      	\draw (2.12,-0.15) -- (2.12,0.15);
		\node[above] at (1.65,0) {$\sqrt{2}k_0$};
    \end{tikzpicture}
  \caption{Full turn} \label{fig:2d-cover-half-2}
  \end{subfigure}
    \caption{
      2D Fourier coverage for a rotating object, incidence direction $\bs=(0,1)$ and measurements taken at $r_2=\rM \in I^+$.
      The Fourier coverage (light red) is a union of infinitely many semicircles, some of which are depicted in red.
      \label{fig:2d-cover-half} }
\end{figure}
\begin{figure}[ht!]
  \centering
  \begin{subfigure}{.25\textwidth}\centering
    \begin{tikzpicture}[scale=0.52]
      \fill[red,opacity=.3] (0,0) arc (90:270:1.5) -- (0,-3) arc (270:360:3) -- (3,0) arc (0:-180:1.5);
      
      \draw[red] (3,0) arc (0:-180:1.5);
      \draw[red] (0,-3) arc (-90:-270:1.5);
      
      \draw[->]   (-3.5,0) -- (3.5,0) node[below,black] {$y_1$};
      \draw[->]   (0,-3.5) -- (0,3.5) node[below left,black] {$y_2$};
      
      \draw (1.5,-0.15) -- (1.5,0.15);
      \node[above] at (1.5,0) {$k_0$};
      
      \draw (3,-0.15) -- (3,0.15);
      \node[above] at (3,0) {$2k_0$};
      
      \draw (-0.15,1.5) -- (0.15,1.5);
      \node[right] at (0,1.5) {$k_0$};
    \end{tikzpicture}
    \caption{Quarter turn}
  \end{subfigure}\hfill
  \begin{subfigure}{.25\textwidth}\centering
  \begin{tikzpicture}[scale=0.52]
    		\fill[red,opacity=.3] (0,0) arc (0:180:1.5) -- (-3,0) arc (180:360:3) -- (3,0) arc (0:-180:1.5);
    		
    		\draw[red] (-3,0) arc (180:0:1.5);
    		\draw[red] (3,0) arc (0:-180:1.5);
    		\draw[red] (0,-3) arc (-90:-270:1.5);
    		
		\draw[->]   (-3.5,0) -- (3.5,0) node[below,black] {$y_1$};
      	\draw[->]   (0,-3.5) -- (0,3.5) node[below left,black] {$y_2$};

		\draw (1.5,-0.15) -- (1.5,0.15);
		\node[above] at (1.5,0) {$k_0$};
		
		\draw (3,-0.15) -- (3,0.15);
		\node[above] at (3,0) {$2k_0$};
		
		\draw (-0.15,1.5) -- (0.15,1.5);
		\node[right] at (0,1.5) {$k_0$};
    \end{tikzpicture}
  \caption{Half turn} \label{fig:2d-max-cover-half}
  \end{subfigure}\hfill
  \begin{subfigure}{.25\textwidth}\centering
  \begin{tikzpicture}[scale=0.52]
    \fill[red,opacity=.3] (0,0) arc (0:180:1.5) -- (-3,0) arc (180:360:3) -- (3,0) arc (0:-180:1.5);
    \fill[red,opacity=.3] (0,0) arc (180:360:1.5) -- (3,0) arc (0:90:3) -- (0,3) arc (90:-90:1.5);
    
    \draw[red] (-3,0) arc (180:0:1.5);
    \draw[red] (3,0) arc (0:-180:1.5);
    \draw[red] (0,-3) arc (-90:-270:1.5);
    \draw[red] (0,3) arc (90:-90:1.5);
    
    \draw[->]   (-3.5,0) -- (3.5,0) node[below,black] {$y_1$};
    \draw[->]   (0,-3.5) -- (0,3.5) node[below left,black] {$y_2$};
    
    \draw (1.5,-0.15) -- (1.5,0.15);
    \node[above] at (1.5,0) {$k_0$};
    
    \draw (3,-0.15) -- (3,0.15);
    \node[above] at (3,0) {$2k_0$};
    
    \draw (-0.15,1.5) -- (0.15,1.5);
    \node[right] at (0,1.5) {$k_0$};
  \end{tikzpicture}
  \caption{Three-quarter turn} \label{fig:2d-max-cover-34}
  \end{subfigure}\hfill
  \begin{subfigure}{.25\textwidth}\centering
    \begin{tikzpicture}[scale=0.52]
    		\fill[red,opacity=.3] (0,0) circle (3);
    		
    		\draw[red] (-3,0) arc (180:0:1.5);
    		\draw[red] (0,3) arc (90:-90:1.5);
    		\draw[red] (3,0) arc (0:-180:1.5);
    		\draw[red] (0,-3) arc (-90:-270:1.5);
    		
		\draw[->]   (-3.5,0) -- (3.5,0) node[below,black] {$y_1$};
      	\draw[->]   (0,-3.5) -- (0,3.5) node[below left,black] {$y_2$};

		\draw (1.5,-0.15) -- (1.5,0.15);
		\node[above] at (1.5,0) {$k_0$};
		
		\draw (3,-0.15) -- (3,0.15);
		\node[above] at (3,0) {$2k_0$};
		
		\draw (-0.15,1.5) -- (0.15,1.5);
		\node[right] at (0,1.5) {$k_0$};
    \end{tikzpicture}
  \caption{Full turn} \label{fig:2d-max-cover-full}
  \end{subfigure}%
    \caption{
      2D Fourier coverage for a rotating object, incidence direction $\bs=(1,0)$ and measurements taken at $r_2=\rM \in I^+$.
      The Fourier coverage (light red) is a union of infinitely many semicircles, some of which are depicted in red.
      \label{fig:2d-max-cover}}
  \end{figure}
  \begin{figure}[ht!]
  \centering
  \includegraphics[trim={3cm 2.6cm 2cm 2.2cm},clip, width=0.33\textwidth]{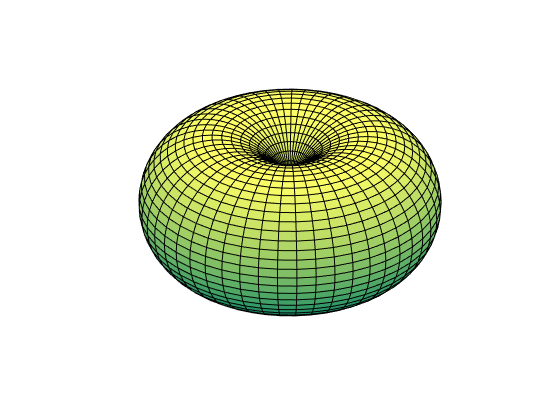} \quad
  \includegraphics[trim={4cm 1.8cm 1cm 2.3cm},clip, width=0.3\textwidth]{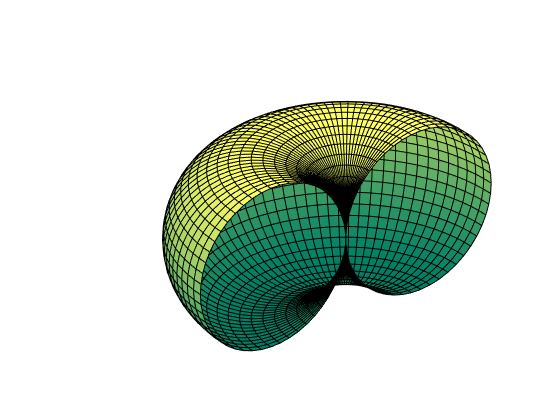} \quad
  \begin{tikzpicture}[scale=0.65]
    \draw[->]   (0,-2.5) -- (0,2.5) node[below left] {$y_1$};
    \draw[->]   (-3.5,0) -- (3.5,0) node[below] {$y_2$};
	\fill[red,opacity=.3] (1.5,0) circle (1.5);
	\fill[red,opacity=.3] (-1.5,0) circle (1.5);
	\draw (1.5,-0.15) -- (1.5,0.15) node[above] {$k_0$};
	\draw (-1.5,-0.15) -- (-1.5,0.15) node[above] {$-k_0$}; 
  \end{tikzpicture}
  \caption{3D Fourier coverage for a full rotation of the object about the $r_1$-axis with incidence direction $\bs=(0, 1, 0)$. 
    \emph{Left and center:} 3D visualization.
    \emph{Right:} 2D cross section through $y_1y_2$-plane. In this case there is no difference in the Fourier coverage between $\rM \in I^+$ or $\rM \in I^-$.
}
  \label{fig:cover3}
\end{figure}

\subsection{Location of measurement hyperplane}\label{sec:measurement}
Consider now moving the measurement hyperplane. It follows from \eqref{eq:recon} that the signed distance $\rM$ from the origin to the hyperplane does not affect the Fourier coverage, at least as long as it stays on one side of the support of~$f$. Therefore, we keep $\rM$ fixed and rotate the measurement hyperplane around the origin according to $R(t)\in SO(d)$. This is equivalent to rotating the incidence direction and the object simultaneously. Denote by $\bs_0$ the original incidence direction. Combining \eqref{eq:recon_n2} for the incidence $\bs(t) = R(t)^\top \bs_0$ with \eqref{eq:recon_n} for the normalized scattering potential $f(R(t) \cdot)$, we obtain
\begin{align*}
	\tilde\ktran u_t(\bx, \rM)
		&= \sqrt{\frac\pi2} \frac{ \i \e^{\pm \i \kappa \rM}}{\kappa} \ktran f \left(R(t)\left(\bh^\pm(\bx) -k_0 R(t)^\top \bs_0 \right) \right) \\
		&= \sqrt{\frac\pi2} \frac{ \i \e^{\pm \i \kappa \rM}}{\kappa} \ktran f \left(R(t) \bh^\pm(\bx) -k_0 \bs_0 \right), \qquad \text{if } \rM \in I^\pm(f\circ R(t)).
\end{align*}
The resulting Fourier coverage is
\begin{equation*}
	\mathcal{Y} = \left\{ R(t) \bh^\pm(\bx) -k_0 \bs_0\in\R^d : |\bx| < k_0, \, 0\le t \le L \right\}.
\end{equation*}
The hemisphere $\bh^\pm$, which is centered at the origin, is rotated \emph{before} it is translated by the fixed vector $-k_0\bs_0$. This means that $\mathcal{Y}\subset\{\by\in\R^d:\abs{\by-k_0\bs_0}=k_0\}$. In contrast to the previous two situations, the coverage is always a set of measure zero. See \autoref{fig:2d-cover-hyperplane}.

\begin{figure}[ht!]\centering
	\begin{tikzpicture}[scale=0.85]
	
	\draw[red] (-1.5,-1.5) arc (180:-90:1.5);
	
	\draw[->]   (-2.5,0) -- (2.5,0) node[below,black] {$y_1$};
	\draw[->]   (0,-3.2) -- (0,0.5) node[right,black] {$y_2$};
	
	\draw[black] (-0.12,-1.5) -- (0.12,-1.5);
	\node[left] at (0,-1.5) {$-k_{0}$};
	\end{tikzpicture}
	\qquad \qquad
	\begin{tikzpicture}[scale=0.85]
	
	\draw[red] (-1.5,-1.5) arc (180:-180:1.5);
	
	\draw[black,->]   (-2.5,0) -- (2.5,0) node[below,black] {$y_1$};
	\draw[black,->]   (0,-3.2) -- (0,0.5) node[right,black] {$y_2$};
	
	\draw[black] (-0.12,-1.5) -- (0.12,-1.5);
	\node[left] at (0,-1.5) {$-k_{0}$};
	\end{tikzpicture}
		\caption{
		2D Fourier coverage for rotating measurement line, starting with the measurement line $r_2=\rM \in I^+$ and rotating it clockwise. Left: Quarter turn. Right: Half turn (or more). Note that the coverage is only the circle, not its interior.}
		\label{fig:2d-cover-hyperplane}
\end{figure}

\subsection{Wave number} \label{sec:cover-frequencies}
We examine how altering the wave number $k_0$ of the incident plane wave affects the Fourier coverage. Denote by $u_t$ the scattered wave generated by the incident field $\ui(\bx)=\e^{\i k_0(t)\, \bx \cdot \bs}$ with wave number $k_0(t)>0$ for $t\in[0,L]$.
We assume that the object's refractive index $n$ and therefore also $f$ does not depend on $k_0(t)$.
Then, according to \eqref{eq:recon}, we have
\begin{equation*}
\tilde \ktran u_t(\bx,\rM)= \sqrt{\frac\pi2} \frac{ \i \e^{\pm \i \kappa(\bx,t) \rM} k_0(t)^2}{\kappa(\bx,t)} \ktran f \left(\bh^\pm(\bx,t) -k_0(t) \bs \right), \qquad \text{if } \rM \in I^\pm.
\end{equation*}
Notice that $\kappa(\bx,t) = \sqrt{k_0(t)^2-|\bx|^2}$ and therefore also $\bh^\pm(\bx,t) = \left(\bx,\pm\kappa(\bx,t)\right)^\top$ depend on $t$ in this case. The Fourier coverage
\begin{equation*}
	\mathcal{Y} = \left\{ \bh^\pm(\bx,t) - k_0(t) \bs \in \R^d : 0\le t \le L, \, |\bx| < k_0(t) \right\}
\end{equation*}
is a union of hemispheres that are translated in direction of $\bs$ and scaled such that each passes through the origin. In contrast to the previous scenarios, there are large missing parts near the origin, see the 2D case in \autoref{fig:2d-cover-frequencies}. This also holds in 3D, where the corresponding Fourier coverages are obtained by rotating those depicted in \autoref{fig:2d-cover-frequencies} around the $y_2$ axis.
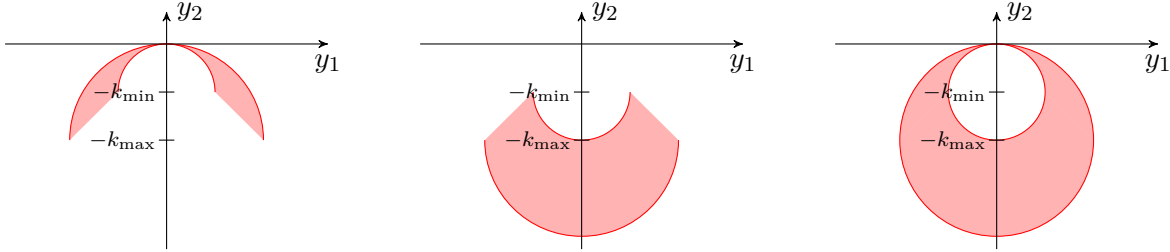
\begin{figure}[ht!]\centering
	\begin{tikzpicture}[scale=0.85]
	\fill[red,opacity=.3] (-1.5,-1.5) arc (180:0:1.5) -- (.75,-.75) arc (0:180:.75);
	
	\draw[red] (-1.5,-1.5) arc (180:0:1.5);
	\draw[red] (-.75,-.75) arc (180:0:.75);
	
	\draw[->]   (-2.5,0) -- (2.5,0) node[below,black] {$y_1$};
	\draw[->]   (0,-3.2) -- (0,0.5) node[right,black] {$y_2$};
	
	\draw[black] (-0.12,-.75) -- (0.12,-.75);
	\node[left] at (0,-0.75) {\scriptsize{$-k_{\mathrm{min}}$}};
	\draw[black] (-0.12,-1.5) -- (0.12,-1.5);
	\node[left] at (0,-1.5) {\scriptsize{$-k_{\mathrm{max}}$}};
	\end{tikzpicture}
	\qquad
	\begin{tikzpicture}[scale=0.85]
	\fill[red,opacity=.3] (-1.5,-1.5) arc (180:360:1.5) -- (.75,-.75) arc (0:-180:.75);
	
	\draw[red] (-1.5,-1.5) arc (180:360:1.5);
	\draw[red] (-.75,-.75) arc (180:360:.75);
	
	\draw[black,->]   (-2.5,0) -- (2.5,0) node[below,black] {$y_1$};
	\draw[black,->]   (0,-3.2) -- (0,0.5) node[right,black] {$y_2$};
	
	\draw[black] (-0.12,-.75) -- (0.12,-.75);
	\node[left] at (0,-0.75) {\scriptsize{$-k_{\mathrm{min}}$}};
	\draw[black] (-0.12,-1.5) -- (0.12,-1.5);
	\node[left] at (0,-1.5) {\scriptsize{$-k_{\mathrm{max}}$}};
	\end{tikzpicture}
	\qquad
	\begin{tikzpicture}[scale=0.85]
	\fill[red,opacity=.3] (-1.5,-1.5) arc (180:0:1.5) -- (.75,-.75) arc (0:180:.75);
	\fill[red,opacity=.3] (-1.5,-1.5) arc (180:360:1.5) -- (.75,-.75) arc (0:-180:.75);
	
	\draw[red] (0,-.75) circle (.75cm);
	\draw[red] (0,-1.5) circle (1.5cm);
	
	\draw[black,->]   (-2.5,0) -- (2.5,0) node[below,black] {$y_1$};
	\draw[black,->]   (0,-3.2) -- (0,0.5) node[right,black] {$y_2$};
	
	\draw (-0.12,-.75) -- (0.12,-.75);
	\node[left] at (0,-0.75) {\scriptsize{$-k_{\mathrm{min}}$}};
	\draw (-0.12,-1.5) -- (0.12,-1.5);
	\node[left] at (0,-1.5) {\scriptsize{$-k_{\mathrm{max}}$}};
	\end{tikzpicture}
	\caption{
		2D Fourier coverage with $\bs=(0,1)$
		where the wave number $k_0(t)$ covers the interval $[k_{\mathrm{min}},k_{\mathrm{max}}]$.
		Left: measurements taken at $r_2=\rM$ with $\rM \in I^+$.
		Center: $\rM \in I^-$.
		Right: both measurements combined.
		\label{fig:2d-cover-frequencies} }
\end{figure}

\subsection{Varying all parameters at once} \label{sec:coverage-general}

Let us assume that the object rotation $R(t) \in SO(d)$, the translation $\bd(t)\in\R^d$, the incidence direction $\bs(t)\in\S^{d-1}$ and the wave number $k_0(t) \in \R_{>0}$ all depend on $t\in[0,L]$. We denote by $\ui_t(\br)=\e^{\i k_0(t) \br\cdot \bs(t)}$ the corresponding incident wave and by
\begin{equation} \label{eq:ut}
  u_t \coloneqq k_0(t)^2 \left( (f\circ\Psi_t)\, \ui_t \right)*G
\end{equation} 
the resulting wave scattered by $f\circ\Psi_t$, see \eqref{eq:motion}.
Analogously to \eqref{eq:recon_n}, we have
\begin{equation} \label{eq:recon_gen}
  \tilde \ktran u_t(\bx,\rM)= \sqrt{\frac\pi2}\, \frac{ \i \e^{\pm \i \kappa(\bx,t) \rM}\, {k_0(t)^2}}{\kappa(\bx,t)}\, \ktran f \left( R(t) \left( \bh^\pm(\bx,t) -k_0(t) \bs(t)\, \right) \right)
  \e^{-\i {\bd(t)}\cdot\left(\bh^\pm(\bx,t) -k_0(t) \bs(t)\right)}
\end{equation}
if $\rM \in I^\pm(f\circ \Psi_t)$, and the respective Fourier coverage is given by
\begin{equation*}
\mathcal{Y} = \left\{ R(t) \left( \bh^\pm(\bx,t) -k_0(t)\, \bs(t) \right) \in \R^d : |\bx| < k_0(t), \, 0\le t \le L \right\}.
\end{equation*}

\begin{remark}[Maximal cover] \label{rem:max-cover}
Assume that $k_0(t)$ has a maximum $k_{\mathrm{max}}.$ Then, under the assumptions of \autoref{sec:coverage-general} the set $\mathcal{Y}$ is always contained in a ball of radius $2k_{\mathrm{max}}.$ In 2D this maximal coverage can be attained when the object makes a full turn and the propagation direction of the plane wave is parallel to the measurement line, see \autoref{fig:2d-max-cover-full}. On the other hand, the fact that $\mathcal{Y}$ is bounded while $\supp \hat f$ is unbounded implies that $f$ cannot be reconstructed exactly using the Fourier diffraction theorem alone \cite{Wol69}.
\end{remark}

\begin{remark}[Redundancy of measurement planes] \label{rem:reflection}
So far we have always considered two options for the location of the measurement hyperplane, $\rM \in I^-$ or $\rM \in I^+$, each leading to a different coverage in general.
The following argument shows that all information obtained at one hyperplane can also be obtained at the other by suitably adapting the incidence direction and the orientation of the object.
For instance, the frequency components of $f$ obtained at $I^-$ (and thus via $\bh^-$) can be accessed at $I^+$ (and via $\bh^+$) when replacing the incidence $\bs_0=(s_1,\dots,s_d)\in\S^{d-1}$ by $(s_1,\dots,s_{d-2},-s_{d-1},-s_d)$ and using the rotation $R_0 = \operatorname{diag}(1,\dots,1,-1,-1)\in SO(d)$,
since we have for $\bx\in\R^{d-1}$ with $\abs\bx<k_0$ that
\[
\bh^+(\bx)-k_0\bs_0
=
R_0 \bigl(\bh^-(x_1,\dots,x_{d-2},-x_{d-1}) - k_0(s_1,\dots,s_{d-2},-s_{d-1},-s_d)\bigr).
\]
\end{remark}
\section{Filtered backpropagation} \label{sec:backprop}
Filtered backpropagation formulae provide an explicit expression for a low-pass filtered approximation of the normalized scattering potential $f$, see \cite{Dev82}, \cite[Sect.\ 6.4.2]{KakSla01} or \cite{KirQueRitSchSet21}. Recall the Fourier coverage  $\mathcal{Y}\subset\R^d$ of the experiment from the previous section.
The filtered backpropagation of $f$ is defined by the Fourier inversion
\begin{equation*}
  f_\mathcal{Y}
  \coloneqq
  \ktran^{-1}(\mathbf{1}_\mathcal{Y}\hat f).
\end{equation*}
If $\hat f$ is integrable on $\mathcal{Y}$, we can express the filtered backpropagation by the integral
\begin{equation} \label{eq:bpp}
f_{\mathcal{Y}}(\br) = (2\pi)^{-\frac{d}{2}}\int_{\mathcal{Y}}\hat f(\by)\e^{\i\by\cdot\br} \dd \by.
\end{equation}
Before applying the Fourier diffraction theorem, \autoref{thm:fdt}, to express the right-hand side in terms of the measurements $u_t(\cdot,\rM)$, the integral is typically transformed into one over $(\bx,t)$. Recall that $\bx=(x_1,\ldots,x_{d-1})$ are the spatial frequencies of the measurements of the scattered wave. This change of coordinates circumvents the irregular sampling in the Fourier domain, which would result from directly discretizing \eqref{eq:recon_gen}.

The following characterization of the filtered backpropagation is a direct consequence of Plancherel's identity, which states that
$
\norm{f}_{L^2(\R^d)}
=
\norm{\smash{\hat f}}_{L^2(\R^d)}
$
for all $f\in L^2(\R^d)$.

\begin{theorem} \label{rem:bp-error}
  Let $f\in L^2(\R^d)$ and the Fourier coverage $\mathcal{Y}\subset\R^d$ be measurable. Then
  \begin{enumerate} \setlength{\parskip}{0pt}  \vspace{-13pt}
  \item $f_{\mathcal{Y}}$ has minimal $L^2$ norm among all functions $g\in L^2(\R^d)$ that satisfy $\hat g = \hat f$ on $\mathcal{Y}$,
  \item $f_{\mathcal{Y}}$ is the $L^2$ best approximation to $f$ in the subspace
  $
    \left\{ g \in L^2(\R^d) : \supp \hat g \subset \mathcal{Y} \right\},
  $
  and
  \item
  if $\mathcal{Y}_1\supset\mathcal{Y}$, then
  $
  \norm{f-f_{\mathcal{Y}_1}}_{L^2(\R^d)}
  \le
  \norm{f-f_{\mathcal{Y}}}_{L^2(\R^d)}.
  $
  \end{enumerate}
\end{theorem}

\subsection{General filtered backpropagation formula}\label{sec:gen-backprop}

We consider the general experiment of \autoref{sec:coverage-general} in which the direction of incidence $\bs(t)\in\S^{d-1}$, the object orientation $R(t)\in SO(d)$ and translation $\bd(t)\in\R^d$, as well as the wave number $k_0(t)>0$ can vary simultaneously depending on $t \in [0,L]$. As pointed out in \autoref{rem:reflection}, we can restrict ourselves to a measurement hyperplane with $\rM \in I^+$ without losing generality. Therefore, we set accordingly $\bh\coloneqq\bh^+$, see \eqref{eq:h}, and define
\begin{equation} \label{eq:U}
\mathcal{U}\coloneqq\left\{(\bx,t)\in\R^d: \abs{\bx} < k_0(t),0\leq t\leq L\right\}.
\end{equation}
The Fourier coverage of the experiment is given by $\mathcal{Y} = T (\mathcal{U})$, where
\begin{equation}\label{eq:T}
T\colon \mathcal{U} \to \mathbb{R}^d, \quad T(\bx,t) \coloneqq R(t)(\bh(\bx,t)-k_0(t)\bs(t)).
\end{equation}
With this notation, we obtain by \eqref{eq:recon_gen} the following relation between the scattered wave $u_t$ and the normalized scattering potential $f$,
\begin{equation} \label{eq:potential}
\tilde\ktran u_t(\bx,\rM)=\sqrt{\frac\pi2} \frac{ \i \e^{\i \kappa \rM}\, {k_0(t)^2}}{\kappa}\,\ktran f(T(\bx,t))
\,\e^{-\i {\bd(t)}\cdot{T(\bx,t)}}.
\end{equation}

\begin{remark}[Experimental setup with discontinuous parameters]
The following backpropagation formula in \autoref{thm:reconstruction} specifically allows $\bs$, $R$ and $k_0$ to be discontinuous functions of $t$. This was done in order to be able to handle experimental setups where one or more of those parameters do not change continuously but only attain a few discrete values, without having to impose unrealistic smoothness assumptions. Imagine, for instance, an object which is subsequently rotated about two different axes, or, one which, during rotation, is subsequently illuminated from a finite number of directions. Such situations can be modeled by a piecewise smooth rotation map $R$ or a piecewise constant $\bs$, respectively. Subsequently, all obtained measurements can be combined into one reconstruction using equation \eqref{eq:reconstruction} below.
\end{remark}

For proving the filtered backpropagation formula, we need the following change of variables formula, which is due to \cite[Thm.\ 2]{Haj93}, see also \cite[Thm.\ 5.8.30]{Bog07}.
The quantity $\operatorname{Card}(T^{-1}(\by))$ is called the \emph{Banach indicatrix} of the map $T$, where $\operatorname{Card}$ denotes the counting measure.

\begin{lemma} \label{lem:substitution}
Let $\Omega\subset \R^d$ be an open set and let $T \colon \Omega \to \R^d$ be a measurable map that has partial derivatives a.e.\ in $\Omega$. 
Denote by ${\det (\nabla T)}$ the determinant of the matrix formed by the partial derivatives of $T$. Suppose, in addition, that $T$ has the Luzin N property, which means that $T$ maps null sets to null sets.
Then, for every measurable set $\mathcal{E} \subset \Omega$ and every measurable function $v\colon \R^d \to \R$, the functions
\begin{equation} \label{eq:2functions}
v(T(\br))\abs{\det(\nabla T(\br))} \mathbf{1}_\mathcal{E}(\br) \quad\text{and}\quad v(\by)\operatorname{Card}(T^{-1}(\by) \cap \mathcal{E})
\end{equation}
are measurable, where we set $v(T(\br))\abs{\det\nabla T(\br)} = 0$ if $v(T(\br))$ is not defined. If one of these functions is integrable, then so is the other and
\begin{equation} \label{eq:substitution}
\int_\mathcal{E} v(T(\br))\abs{\det(\nabla T(\br))}\dd \br 
= \int_{\R^d} v(\by)\operatorname{Card}(T^{-1}(\by) \cap \mathcal{E})\dd \by.
\end{equation}
\end{lemma}

\begin{theorem}[Filtered backpropagation formula]\label{thm:reconstruction}
Let $f\in L^1(\R^d)$ have compact support and $L>0$. Assume that each of the maps $\bs\colon [0,L] \to \mathbb{S}^{d-1}$, $R\colon [0,L] \to SO(d)$, $\bd\colon [0,L] \to \R^d$ and $k_0\colon [0,L] \to (0,+\infty)$ is piecewise $C^1$ in every component, i.e., except at finitely many points $t_1,\ldots,t_m \in [0,L]$,
and that $\bs'$, $R'$, and $k_0'$ are bounded. 
Let $u_t$ be defined as in \eqref{eq:ut} and assume that $\rM \in I^+(f \circ \Psi_t)$ for all $t\in[0,L]$.
Then, $f_{\mathcal{Y}}(\br)$ is finite for all $\br \in \R^d$ and
\begin{equation}\label{eq:reconstruction}
	f_{\mathcal{Y}}(\br) = 2(2\pi)^{-\frac{1+d}{2}}  \int_{\mathcal{U}} \frac{\kappa(\bx,t) \, \e^{\i T(\bx ,t) \cdot (\br+\bd(t))}\, \tilde\ktran u_t(\bx,\rM) \abs{\det \left(\nabla T(\bx,t) \right)} }{k_0(t)^2\, \i \e^{\i \kappa(\bx,t) \rM} \operatorname{Card}(T^{-1}(T(\bx,t)))} \dd(\bx,t),
\end{equation}
where 
the Jacobian determinant of $T$, as defined in \eqref{eq:T}, is given by
\begin{equation}\label{eq:jacobian}
	\det \left(\nabla T(\bx,t)\right) = \frac{k_0(t) k_0'(t) - R(t)\bh(\bx,t) \cdot (k_0(t) R(t) \bs(t))'}{\kappa(\bx,t)}.
\end{equation}
\end{theorem}

\begin{proof}
  We note that
  \begin{equation*}\label{eq:2kmax}
    \abs{T(\bx,t)} \le 2 \sup\{k_0(t): t\in[0,L]\}
  \end{equation*}
  for all $(\bx,t)\in\mathcal{U}$, cf.\ \autoref{rem:max-cover}.	As $k_0'$ and therefore $k_0$ is bounded, the set $\mathcal{Y}$ is bounded. Therefore, $f_{\mathcal{Y}}(\br)$ is finite for every $\br\in\R^d$. We will prove the theorem through usage of the change of variables formula in \autoref{lem:substitution} with $\Omega=\R^d$, $\mathcal{E}=\mathcal{U}$ and $v=\mathbf{1}_{T(\mathcal{U})}$.
  To this end, we need to show that $T$ from \eqref{eq:T} fulfills the prerequisites of \autoref{lem:substitution}.
  By assumption, $T$ has partial derivatives a.e.\ on $\mathcal{U}$, and we set $T$ to zero on $\R^d\setminus\mathcal{U}$ to obtain this property on whole $\Omega=\R^d$. According to \cite[Lem.\ 7.25]{Rud87}, differentiable maps from $\R^d$ into $\R^d$ have the Luzin N property. Now let $\mathcal{L}$ denote the $d$-dimensional Lebesgue measure and consider $E\subset\mathcal{U}$ with $\mathcal{L}(E)=0$. Let $D = \{(\bx,t) \in \mathcal{U} : t \in\{0,t_1,\ldots,t_m,L\}\}$ be the set where $T$ might not be $C^1$. By decomposing $E=(E\cap D)\cup (E\cap D^\mathrm{c})$ with $D^\mathrm{c}$ being the complement of $D$ in $\mathcal{U}$, we obtain
	\begin{align*}
	\mathcal{L}(T(E))&=\mathcal{L}(T((E\cap D)\cup (E\cap D^\mathrm{c})))=\mathcal{L}(T(E\cap D)\cup T(E\cap D^\mathrm{c}))\\
	&\leq\mathcal{L}(T(E\cap D)) +\mathcal{L}(T(E\cap D^\mathrm{c}))\leq\mathcal{L}(T(D)) +\mathcal{L}(T(E\cap D^\mathrm{c}))=0
	\end{align*}
	as $T(D)$ is a finite union of hypersurfaces, recall \eqref{eq:hemisphere}, and $T$ is $C^1$ on $D^\mathrm{c}$.
  Therefore, $T$ has the Luzin N property.
	
	Next we show that the left-hand side of \eqref{eq:2functions} is integrable, which is equivalent to $\det(\nabla T)\in L^1(\mathcal{U})$. For almost every $t \in [0,L]$ the Jacobian matrix of $T$ is given by
	\begin{equation*}
	\nabla T
		= \begin{pmatrix} \frac{\partial T}{\partial x_1} & \cdots & \frac{\partial T}{\partial x_{d-1}} & \frac{\partial T}{\partial t} \end{pmatrix},
	\end{equation*}
	where
	\begin{align*}
	\frac{\partial T}{\partial k_i}
		&= R \frac{\partial \bh}{\partial x_i} = R \left( \be_i - \frac{x_i}{\kappa} \be_d \right), \\
	\frac{\partial T}{\partial t}
		&= R \frac{\partial \bh}{\partial t} + R'\bh - \left( k_0 R \bs \right)' = \frac{k_0k_0'}{\kappa} R \be_d + R'\bh - \left( k_0 R \bs \right)'
	\end{align*}
  and $\be_i$ denotes the $i$-th unit vector in $\R^d$.
	Therefore, its determinant can be expressed as
	\begin{align*}
		\det \left( \nabla T\right) = \det \left( R^\top \nabla T\right) = 
		\det \left(
		\begin{array}{c|c}
			I_{d-1} & \multirow{2}{*}{$\frac{k_0k_0'}{ \kappa} \be_d + \bv$} \\ 
			-\bx^\top/\kappa & 
		\end{array}
		\right) 	=
		\frac{k_0k_0'}{ \kappa} +
		\det \left(
		\begin{array}{c|c}
			I_{d-1} & \multirow{2}{*}{$\bv$} \\ 
			-\bx^\top/\kappa & 
		\end{array}
		\right),
	\end{align*}
	where $I_{d-1}$ is the identity matrix of size $d-1$ and $\bv = R^\top (R'\bh - \left( k_0 R \bs \right)')$. For the determinant of a $2 \times 2$ block matrix with invertible upper left block we recall that
	\begin{align*}
	\det \begin{pmatrix} A & B \\ C & D 	\end{pmatrix}
	= \det \left( \begin{pmatrix} A & 0 \\ C & I 	\end{pmatrix} \begin{pmatrix} I & A^{-1}B \\ 0 & D- CA^{-1}B 	\end{pmatrix} \right)
	= \det(A) \det(D- CA^{-1}B).
	\end{align*}
	It follows that
	\begin{align*}
		\det \left(
		\begin{array}{c|c}
			I_{d-1} & \multirow{2}{*}{$\bv$} \\ 
			-\bx^\top/\kappa & 
		\end{array}
		\right)
		=
		v_d - \frac{\bx^\top \overline{\bv}}{- \kappa} = \frac{\bh \cdot \bv}{ \kappa}.
	\end{align*}
	In total the Jacobian determinant equals
	\begin{equation*}
	\det \left( \nabla T\right) = \frac{k_0k_0' + \bh \cdot \bv}{ \kappa}.
	\end{equation*}		
	Due to the stated assumptions on $k_0$, $R$ and $\bs$ the numerator is bounded. Therefore, the determinant is integrable on $\mathcal{U}$, if $1/\kappa$ is. Recalling \eqref{eq:kappa_L1loc}, we see that
	\begin{align*}
		\norm{\frac{1}{\kappa}}_{L^1(\mathcal{U})}
			= \int_{\mathcal{U}} \frac{\dd (\bx, t) }{\abs{\kappa}} 
			&\le \abs{\S^{d-2}} \int_0^L k_0(t)^{d-\frac52} \int_0^{k_0(t)} \frac{\dd \rho}{\sqrt{\abs{k_0(t) - \rho}}} \dd t
			\\&
      = 2 \abs{\S^{d-2}} \int_0^L k_0(t)^{d-2} \dd t,
	\end{align*}
  which is finite since $k_0$ is bounded.
	We conclude that the assumptions of \autoref{lem:substitution} are fulfilled and in particular \eqref{eq:substitution} with $\mathcal{E}=\mathcal{U}$ and $v=\mathbf{1}_{T(\mathcal{U})}$ yields
	\begin{equation*}
	\int_{T(\mathcal{U})}\operatorname{Card}(T^{-1}(\by))\dd \by=\int_\mathcal{U} \abs{\det(\nabla T(\bx,t))} \dd(\bx,t).
	\end{equation*}
	Therefore, $\operatorname{Card}(T^{-1}(\by))<\infty$ for a.e.\ $\by \in T(\mathcal{U})$. As $\operatorname{Card}(T^{-1}(\cdot))>0$ on $T(\mathcal{U})$, we may then write
	\begin{equation*}
	f_{\mathcal{Y}}(\br)=(2\pi)^{-\frac{d}{2}}\int_{T(\mathcal{U})}\e^{\i\by\cdot\br}\ktran f(\by)\frac{\operatorname{Card}(T^{-1}(\by))}{\operatorname{Card}(T^{-1}(\by))} \dd \by.
	\end{equation*}
	Invoking \autoref{lem:substitution} again, where we now integrate the function $$v(\by) =\begin{cases} \e^{\i \by \cdot \br} \ktran f(\by)/\operatorname{Card}(T^{-1}(\by)), & \by\in T(\mathcal{U}), \\0, &\text{otherwise},\end{cases}$$ gives
	\begin{equation} \label{eq:approx2g}
		f_{\mathcal{Y}}(\br)
	= (2\pi)^{-\frac{d}{2}} \int_{\mathcal{U}} \e^{\i T(\bx,t) \cdot \br} \ktran f(T(\bx,t)) \frac{\abs{\det (\nabla T(\bx,t))}}{\operatorname{Card}(T^{-1}(T(\bx,t)))}  \, \dd (\bx,t).
	\end{equation}
	By \eqref{eq:potential}, we can express $\ktran f$ in terms of the measurements and \eqref{eq:reconstruction} is then established.
	
	It remains to verify \eqref{eq:jacobian}. We have already shown that
	\begin{equation*}
		\det \left( \nabla T\right) = \frac{k_0k_0' + \bh \cdot \left( R^\top (R'\bh - \left( k_0 R \bs \right)' ) \right)}{ \kappa}.
	\end{equation*}
	In order to finish the calculation, we only have to observe that the matrix $R^\top R'$ is skew-symmetric, which can be seen by differentiating the identity $R^\top R = I_d.$ Therefore, $\by \cdot R^\top R' \by = 0$ for all $\by \in \R^d$.
\end{proof}

\subsection{Non-absorbing object} \label{sec:non-absorbing}

In many situations, such as optical diffraction tomography of biological cells, the refractive index $n$ and therefore the normalized scattering potential $f$ are assumed to be real-valued, which means that absorption is neglected, cf.\ \cite{BeiQue22,MulSchGuc15}.
Then the Fourier transform of $f\colon \R^d\to\R$ is conjugate symmetric, 
\begin{equation} \label{eq:Friedel}
\ktran f(\by)
=
\overline{\ktran f(-\by)}
,\qquad \forall\by\in\R^d,
\end{equation}
also known as Friedel's law,
where $\overline{z}$ denotes the complex conjugate of $z\in\C$. The reconstruction $f_{\mathcal{Y}}$ does not account for this symmetry. It might even happen that $f_{\mathcal{Y}}$ has a non-vanishing imaginary part despite the fact that $f$ is real-valued.

By \eqref{eq:Friedel}, we obtain from the measurements the Fourier transform $\ktran f$ not only on $\mathcal Y$, but also on $-\mathcal{Y} = \{-\by : \by\in\mathcal Y\}$,
and therefore the extended Fourier coverage
\begin{equation} \label{eq:Ysym}
  \mathcal{Y}_\sym \coloneqq 
  \mathcal{Y}\cup (-\mathcal{Y}).
\end{equation}
Analogously to \autoref{rem:bp-error}, the backpropagation $f_{\mathcal Y_\sym}$ minimizes $\norm{g}_{L^2(\R^d)}$ among all real-valued functions $g\in L^2(\R^d)$ that satisfy $\ktran g = \ktran f$ on $\mathcal{Y}$. In order to provide a backpropagation formula for $f_{\mathcal{Y}_\sym}$ similar to \eqref{eq:reconstruction}, we set 
\begin{equation*}
  \mathcal{U}_\sym\coloneqq\left\{(\bx,t)\in\R^d: \abs{\bx} < k_0(\abs{t}),-L\leq t\leq L\right\},
\end{equation*}
and we replace the coordinate transformation $T$ of \eqref{eq:T} by
\begin{equation} \label{eq:Tsym}
  T_\sym\colon \mathcal{U}_\sym \to \mathbb{R}^d, \quad T_\sym(\bx,t) \coloneqq \sgn(t) \, T(\bx,\abs{t}),
\end{equation}
with the sign function
\begin{equation} \label{eq:sgn}
  \sgn(t)
  \coloneqq
  \begin{cases}
    \frac{t}{\abs{t}}, & t\neq0,\\
    0, & t=0.
  \end{cases}
\end{equation}
Here, a negative $t$ is associated with the reflected points $- T_\sym(\bx,-t)$ in Fourier space.

\begin{theorem}[Filtered backpropagation with non-absorbing object] \label{thm:reconstruction2}
  Let the assumptions of \autoref{thm:reconstruction} be satisfied. In addition, assume that $f$ is real-valued.
  Then
  \begin{equation} \label{eq:reconstruction2}
    f_{\mathcal{Y}_\sym}(\br)
    = 4(2\pi)^{-\frac{d+1}{2}} 
    \operatorname{Re} \left( \int_{\mathcal{U}}
    \frac{\kappa\, \e^{\i T(\bx,t) \cdot (\br+\bd)} \abs{\det(\nabla T(\bx,t))} \tilde\ktran u_t(\bx,\rM)}
    {k_0(t)^2\, \i \e^{\i \kappa \rM}\, \operatorname{Card}(T_\sym^{-1}( T(\bx,t)))}
    \dd (\bx,t) \right),
  \end{equation}
  where $\operatorname{Re}$ denotes the real part, $u_t$ is given in \eqref{eq:ut}, and $\det(\nabla T)$ in \eqref{eq:jacobian}.
\end{theorem}

\begin{proof}
  Since $T_\sym$ satisfies the same assumptions as $T$ in the proof of \autoref{thm:reconstruction} with the points of possible non-smoothness $\{-t_m, \dots, -t_1, 0, t_1, \dots, t_m\}$,
  we obtain analogously to the derivation of \eqref{eq:approx2g} in the proof of \autoref{thm:reconstruction} that
  \begin{equation*} 
  f_{\mathcal{Y}_\sym}(\br)
    = (2\pi)^{-\frac{d}{2}} \int_{ \mathcal{U}_\sym} \e^{\i T_\sym(\bx,t) \cdot \br} \ktran f(T_\sym(\bx,t)) \frac{\abs{\det (\nabla T_\sym(\bx,t))}}{\operatorname{Card}(T_\sym^{-1}(T_\sym(\bx,t)))} \dd (\bx,t).
  \end{equation*}
  Splitting up the domain of integration $\mathcal{U}_\sym = \mathcal{U} \cup \{(\bx,t): \abs\bx<k_0(\abs t),t\in[-L,0]\}$, we obtain
  \begin{align*}
    f_{\mathcal{Y}_\sym}(\br)
    ={}& (2\pi)^{-\frac{d}{2}} \int_{{ \mathcal{U}}} \e^{\i T(\bx,t) \cdot \br} \ktran f( T(\bx,t)) \frac{\abs{\det (\nabla T(\bx,t))}}{\operatorname{Card}(T_\sym^{-1}( T(\bx,t)))}  \dd (\bx,t)
    \\
    & + (2\pi)^{-\frac{d}{2}} \int_{{ \mathcal{U}}} \e^{-\i T(\bx,t) \cdot \br} \ktran f(- T(\bx,t)) \frac{\abs{\det (\nabla T(\bx,t))}}{\operatorname{Card}(T_\sym^{-1}(- T(\bx,t)))}  \dd (\bx,t),
  \end{align*}
  where we have used the substitution $t\mapsto -t$ and the property $T_\sym(\bx,-t) = - T_\sym(\bx,t)$ in the second integral.
  This property also implies that $T_\sym^{-1}(\by)$ is isomorphic to $T_\sym^{-1}(-\by)$ and therefore $\operatorname{Card}(T_\sym^{-1}( \by))
  = \operatorname{Card}(T_\sym^{-1}(-\by))$ for every $\by$. It now follows from \eqref{eq:Friedel} that the second integral is the complex conjugate of the first so that
  \begin{align*}
  	f_{\mathcal{Y}_\sym}(\br)
    ={}& 2(2\pi)^{-\frac{d}{2}} \operatorname{Re} \left( \int_{{ \mathcal{U}}} \e^{\i T(\bx,t) \cdot \br} \ktran f( T(\bx,t)) \frac{\abs{\det (\nabla T(\bx,t))}}{\operatorname{Card}(T_\sym^{-1}( T(\bx,t)))}  \dd (\bx,t) \right).
  \end{align*}
  Using the Fourier diffraction theorem in \eqref{eq:potential} finishes the proof.
\end{proof}

\begin{remark}[Comparison of the backpropagation formulae]
  The filtered backpropagation formula with symmetrization \eqref{eq:reconstruction2}
  differs from \eqref{eq:reconstruction} in
  that we take twice the real part and we compute the Banach indicatrix of $T_\sym$.
  For real-valued $f$, we can compare the two reconstructions $f_{\mathcal{Y}_\sym}$ and $f_{\mathcal{Y}}$.
  By \autoref{rem:bp-error}, we always have $\|f-f_{\mathcal{Y}_\sym}\|_{L^2(\R^d)}\le\|f-f_{\mathcal{Y}}\|_{L^2(\R^d)}$.
  If $\mathcal Y$ is point symmetric with respect to the origin, i.e.\ $\mathcal Y = -\mathcal Y$, then both yield the same result.
  Otherwise, $f_{\mathcal{Y}}$ might have a non-vanishing imaginary part, but even considering only the real part is not ideal.
  In the extreme case where $\mathcal Y \cap (-\mathcal Y)$ is a null set, as in \autoref{fig:2d-max-cover-half},
  we obtain $f_{\mathcal{Y}_\sym} = 2\operatorname{Re}(f_{\mathcal{Y}})$,
  so the reconstruction with \eqref{eq:reconstruction2} is considerably better.
\end{remark}

\subsection{Filtered backpropagation with multi-dimensional parameter set}
\label{sec:d-parameter}
For 3D angle scanning, cf.\ \autoref{sec:incidence}, one option is to move the incidence along a two-dimensional set.
In order to handle such an experiment, we extend the filtered backpropagation of \autoref{thm:reconstruction} by making $t\in[0,L]$ a multi-dimensional parameter $\bt\in\mathcal{A}\subset\R^{q+1}$ with $q\in\N$.
We substitute $\bx\in\mathcal{B}^{d-1}_{k_0}$ by 
$\bv = \frac{\bx}{\abs\bx} \arcsin\frac{\abs\bx}{k_0}\in\mathcal{B}^{d-1}_{\pi/2}$ if $\bx\neq\bo$.
Then we have $\bx = k_0 \frac{\bv}{\abs{\bv}} \sin\abs\bv$
and 
$\kappa
= k_0 \cos\abs\bv$.
Accordingly, we replace the transformation $T$ in \eqref{eq:T} by
\begin{equation*}
  U \colon \mathcal{B}^{d-1}_{\pi/2} \times \mathcal{A}\to\R^d,\quad
  U(\bv,\bt)\coloneqq
  k_0(\bt)\, R(\bt) \left( \begin{psmallmatrix}
    \bv \frac{\sin\abs\bv}{\abs\bv}\\ \cos\abs\bv
  \end{psmallmatrix} - \bs(\bt)\right).
\end{equation*}
This parameter change makes $U$ Lipschitz, as opposed to $T$.
For a set $ S\subset \R^{d+q}$,
we define  $\operatorname{diam}(S) \coloneqq \sup\{\abs{\bu-\bv} : \bu,\bv\in S\}$
and the $q$-dimensional Hausdorff measure
$$
H^q(S)
\coloneqq \sup_{\delta>0} \left(\inf \left\{ \sum_{i=1}^{\infty} \frac{\pi^{q/2} \operatorname{diam}(B_i)^q}{\Gamma(\frac q2 +1)\, 2^q} : \bigcup_{i=1}^{\infty} B_{i} \supset S,\, B_i\subset\R^{d+q},\, \operatorname{diam}(B_i) < \delta \right\} \right).
$$ 
\begin{theorem} \label{thm:reconstruction-nd}
  Let $\mathcal{A}\subset\R^{q+1}$ be a bounded, open set and each of the maps $R\colon \mathcal{A}\to SO(d)$, $\bs\colon \mathcal{A}\to\S^{d-1}$, $\bd\colon \mathcal{A}\to\R^d$, and $k_0\colon \mathcal{A}\to(0,+\infty)$ be $C^1$ with bounded partial derivatives.
  Further let $f\in L^1(\R^d)$ have compact support,
  $u_\bt$ be defined as in \eqref{eq:ut} and $\rM \in I^+(f \circ \Psi_\bt)$ for all $\bt\in\mathcal{A}$.
  Denote by $|\nabla U|$ the square root of the sum of the squares of the determinants of the $d\times d$ minors of the Jacobian of $U$.
  With $\mathcal{W} \coloneqq \{ \bz \in \R^d : H^q(U^{-1}(\bz)) > 0\}$, we have for all $\br\in\R^d$ 
  \begin{equation*} \label{eq:bpp_nd}
    f_{\mathcal{W}}(\br) =
    2(2\pi)^{-\frac{d+1}{2}}
    \int_{\mathcal{B}^{d-1}_{\pi/2} \times \mathcal{A}} \frac{\cos\abs\bv\, \e^{\i U(\bv,\bt)\cdot(\br+\bd(\bt))} \tilde\ktran u_\bt(\bv\frac{\sin\abs\bv}{\abs\bv},\rM) \abs{\nabla U(\bv,\bt)}}{\i \e^{\i k_0(\bt) \rM \cos\abs\bv}\, k_0(\bt)\, H^q(U^{-1}(U(\bv,\bt)))} 
    \dd(\bv,\bt).
  \end{equation*}
\end{theorem}
\begin{proof}
  We first show that $U$ is Lipschitz.
  All partial derivatives of $U$ with respect to $\bt$ are bounded by assumption.
  Since $\abs\bv^{-1}{\sin\abs\bv} \le 1$ for all $\bv\in\R^{d-1}\setminus\{\bo\}$,
  we see that 
  \begin{equation*}
    \frac{\partial U(\bv,\bt)}{\partial v_j} 
    =
    k_0 R(\bt) \begin{pmatrix}
      \frac{\sin\abs\bv}{\abs\bv} \be^j + \bv \frac{v_j}{\abs\bv} \left(\frac{\cos\abs\bv}{\abs\bv} - \frac{\sin\abs\bv}{\abs\bv^2}\right)\\
      - \frac{\bv}{\abs\bv}\, \sin\abs\bv
    \end{pmatrix}
    ,\qquad \forall j=1,\dots,d-1,
  \end{equation*}
  is uniformly bounded,
  which implies that $U$ is Lipschitz.
  The coarea formula \cite[Thm.\ 3.2.12]{Fed96}, see also \cite{MalSwaZie02},
  states for any $g\in L^1(\B^{d-1}_{\pi/2} \times \mathcal{A})$ and Lipschitz-continuous $U$ that
  \begin{equation} \label{eq:coarea}
    \int_{\B^{d-1}_{\pi/2} \times \mathcal{A}} g(\bv,\bt) \abs{\nabla U(\bv,\bt)} \dd(\bv,\bt)
    =
    \int_{\R^d} \int_{U^{-1}(\bz)} g(\bv,\bt) \dd H^{q}(\bv,\bt) \dd \bz.
  \end{equation}
  Plugging into \eqref{eq:coarea} the indicator function of some $A\subset \mathcal{B}^{d-1}_{\pi/2} \times \mathcal{A}$ with $H^q(U^{-1}(U(A)))=0$
  yields $\int_A \abs{\nabla U(\bv,\bt)} \dd(\bv,\bt)=0$, and therefore $\abs{\nabla U}$ vanishes a.e.\ on $A$.
  Hence \eqref{eq:coarea} remains valid when the left integral is restricted to $S_0 \coloneqq \supp(H^q(U^{-1}\circ U))$.
  
  Let $\varepsilon>0$. We define the set 
  $
  S_\varepsilon \coloneqq \{ (\bv,\bt) \in \B^{d-1}_{\pi/2}\times\mathcal{A} : {H^q(U^{-1}(U(\bv,\bt)))} > \varepsilon \}
  $
  and the function
  $$
  g_\varepsilon(\bv,\bt)
  \coloneqq
  \begin{dcases}
    \frac{1}{H^q(U^{-1}(U(\bv,\bt)))}, & (\bv,\bt)\in S_\varepsilon,\\
    0, & \text{otherwise},
  \end{dcases}
  $$
  which is integrable on $\B^{d-1}_{\pi/2}\times\mathcal{A}$.
  Inserting $g_\varepsilon$ into the coarea formula \eqref{eq:coarea} yields
  \begin{align*} 
    \int_{\mathcal{B}^{d-1}_{\pi/2} \times \mathcal{A}} g_\varepsilon(\bv,\bt) {\abs{\nabla U(\bv,\bt)}} \dd(\bv,\bt)
    &=
    \int_{\R^d} \int_{U^{-1}(\bz)} g_\varepsilon(\bv,\bt) \dd H^{q}(\bv,\bt) \dd \bz\\&
    \le
    \int_{\mathcal{W}} \frac{1}{H^q(U^{-1}(\bz))} \int_{U^{-1}(\bz)} \dd H^{q}(\bv,\bt) \dd \bz
    \le
    \abs{\B^d_{2k_\mathrm{max}}}
  \end{align*} 
  because $\mathcal{W} \subset \B^d_{2k_\mathrm{max}}$ by \autoref{rem:max-cover}.
  Since the right-hand side is bounded independently of $\varepsilon$, we see that ${\abs{\nabla U}} / {H^q(U^{-1}\circ U)}$ is integrable on
  $\bigcup_{\varepsilon>0} S_\varepsilon = S_0$.
  
  Let $\br\in\R^d$. Then
  $$
  a_\varepsilon(\bv,\bt) \coloneqq {\e^{\i U(\bv,\bt)\cdot\br} \ktran f(U(\bv,\bt))} g_\varepsilon(\bv,\bt),
  $$ 
  is in $L^1(\B^{d-1}_{\pi/2} \times \mathcal{A})$ because $\mathcal Ff$ is bounded.
  Defining $\mathcal{W}_\varepsilon \coloneqq U^{-1}(S_\varepsilon)$
  and inserting $a_\varepsilon$ into \eqref{eq:coarea}, we obtain
  \begin{align*}
    \int_{S_0} a_\varepsilon(\bv,\bt) \abs{\nabla U(\bv,\bt)} \dd(\bv,\bt)
    &=
    \int_{\mathcal{W}} \frac{\e^{\i \bz\cdot\br} \ktran f(\bz)}{H^q(U^{-1}(\bz))} \mathbf{1}_{\mathcal{W}_\varepsilon}(\bz) \int_{U^{-1}(\bz)}  \dd H^q(\bv,\bt) \dd \bz
    \\
    &=
    \int_{\mathcal{W}} {\e^{\i \bz\cdot\br} \ktran f(\bz)} \mathbf{1}_{\mathcal{W}_\varepsilon}(\bz) \dd\bz.
  \end{align*}
  The integrand on the left has the integrable upper bound $\abs{\ktran f} {\abs{\nabla U}} / {H^q(U^{-1}\circ U)}$, and
  the integrand on the right is bounded by $\abs{\ktran f}$.
  Applying Lebesgue's dominated convergence theorem for $\varepsilon\to0$ on both sides yields
  \begin{equation*}
    \int_{S_0} \frac{\e^{\i U(\bv,\bt)\cdot\br} \ktran f(U(\bv,\bt))}{H^q(U^{-1}(U(\bv,\bt)))} \abs{\nabla U(\bv,\bt)} \dd(\bv,\bt)
    =
    \int_{\mathcal{W}} {\e^{\i \bz\cdot\br} \ktran f(\bz)} \dd\bz.
  \end{equation*}
  Together with \eqref{eq:recon_gen} and $\kappa = k_0(\bt) \cos\abs\bv$, this shows the assertion.
\end{proof}

\begin{remark}
The backpropagation formula of \autoref{thm:reconstruction}, if all the parameters are $C^1$, can be seen as a special case of \autoref{thm:reconstruction-nd}.
We note that for $q=0$, the set $\mathcal{W}$ coincides with $\mathcal{Y}$, while for $q\ge1$ we only know the inclusion $\mathcal{W}\subset\mathcal{Y}$.
Furthermore, \autoref{thm:reconstruction-nd} requires a different parameterization of the Fourier cover $\mathcal{Y}$.
\end{remark}

\subsection{Special cases}\label{sec:examples}

Below we give a few examples of the filtered backpropagation formulae provided by \autoref{thm:reconstruction} and \autoref{thm:reconstruction2}.

\begin{example}[Object rotation in 2D]\label{ex:2d}
Consider the 2D transmission setup with incidence direction $\bs = (0,1)^\top$, measurement line $r_2 = \rM \in I^+$ and fixed wave number $k_0$. Assuming that the object makes a full turn according to 
\begin{equation*}
	R(t) = \begin{pmatrix} \cos t & - \sin t \\ \sin t & \phantom{-}\cos t \end{pmatrix},\qquad t \in [0,2\pi],
\end{equation*}
the filtered backpropagation formula \eqref{eq:reconstruction} reduces to the well-known
\begin{equation}\label{eq:2d-transmission-bp}
	f_{\mathcal{B}_{\sqrt{2}k_0}}(\br) = \frac{-\i}{k_0} (2\pi)^{-3/2} \int_0^{2\pi} \int_{-k_0}^{k_0} \e^{\i T(x,t) \cdot \br -\i\kappa \rM} \mathcal{F}_1u_t(x,\rM) \abs{x} \dd x \dd t,
  \qquad\text{for all } \br \in \mathbb{R}^2.
\end{equation}
See also \cite{Dev82,KakSla01,Sla85}. The Fourier coverage of this experiment is a disk of radius $\sqrt{2}k_0$ as depicted in \autoref{fig:2d-cover-half-2}.

Changing the incidence direction to $\bs = (1,0)^\top$ leads to a disk of radius $2k_0$, cf.\ \autoref{fig:2d-max-cover-full}. This is the largest possible coverage for the given wave number, as discussed in \autoref{rem:max-cover}. The resulting reconstruction
\begin{equation*}
	f_{\mathcal{B}_{2k_0}}(\br) = \frac{-2\i}{k_0} (2\pi)^{-3/2} \int_0^{2\pi} \int_{-k_0}^{k_0} \e^{\i T(x,t) \cdot \br -\i\kappa \rM} \mathcal{F}_1u_t(x,\rM) \kappa \dd x \dd t,
  \qquad\text{for all } \br \in \mathbb{R}^2,
\end{equation*}
has a smaller $L^2$ approximation error than the one given in \eqref{eq:2d-transmission-bp} according to \autoref{rem:bp-error}.

If $f$ is real-valued, then a half turn of the object is actually enough to recover $f_{\mathcal{B}_{2k_0}}$. This is due to Friedel's law \eqref{eq:Friedel} and the fact that the coverage $\mathcal{Y}$ for a half turn, corresponding to $t\in[0,\pi]$, see \autoref{fig:2d-max-cover-half}, satisfies $\mathcal{Y}_\sym = \mathcal{Y} \cup (-\mathcal{Y}) = \mathcal{B}_{2k_0}$. The symmetrized backpropagation formula from \autoref{thm:reconstruction2} gives
\begin{equation*}
	f_{\mathcal{B}_{2k_0}}(\br) = \frac{-4}{k_0} (2\pi)^{-3/2} \operatorname{Re} \left( \int_0^{\pi} \int_{-k_0}^{k_0} -\i \e^{\i T(x,t) \cdot \br -\i\kappa \rM} \mathcal{F}_1u_t(x,\rM) \kappa \dd x \dd t \right) ,
  \qquad\text{for all } \br \in \mathbb{R}^2.
\end{equation*}
\end{example}

\begin{example}[2D angle scan]
\label{ex:angle+rot}
Consider an experimental setup of angle scanning as in \autoref{fig:rot_wave-2d} center, which is repeated for the object rotated by $90\,^\circ$.
With the measurement line $r_2=\rM\in I^+$ and wave number $k_0$, we set the incidence $\bs(t)=(\cos t,\sin t)$ for $t\in[0,2\pi]$ and the piecewise constant rotation $R(t)=\pm I$ if $t \gtrless \pi$.
Up to zero sets,
the Fourier coverage is the union of four disks of radius $k_0$, namely 
\[\mathcal{Y}= 
\B^2_{k_0}\!\begin{psmallmatrix}0\\k_0\end{psmallmatrix}\cup
\B^2_{k_0}\!\begin{psmallmatrix}0\\-k_0\end{psmallmatrix} \cup 
\B^2_{k_0}\!\begin{psmallmatrix}k_0\\0\end{psmallmatrix}\cup 
\B^2_{k_0}\!\begin{psmallmatrix}-k_0\\0\end{psmallmatrix}.\]
Any point in $\mathcal{Y}$ is contained either in one or in two of these disks, therefore the Banach indicatrix is given for almost every $\by\in\mathcal{Y}$ by
\begin{equation*}
\operatorname{Card}(T^{-1}(\by))
=
\begin{cases}
  2, & \text{if }
  k_0-\sqrt{k_0^2-y_1^2} 
  <\abs{y_2}
  <\sqrt{k_0^2-(\abs{y_1}-k_0)^2},\\
  1, & \text{otherwise},
\end{cases}
\end{equation*}
see \autoref{fig:angle+rot} left.
The backpropagation formula \eqref{eq:reconstruction} becomes
\begin{equation*}
f_{\mathcal{Y}}(\br) = \frac{-2\i(2\pi)^{-\frac{3}{2}}}{k_0}  \int_{0}^{2\pi} \int_{-k_0}^{k_0} \frac{\e^{\i T(x ,t) \cdot \br - \i \kappa(x) \rM}\, \ktran_1 u_t(x,\rM) }{\operatorname{Card}(T^{-1}(T(x,t)))} \abs{\kappa(x) \cos t-x\sin t} \dd x \dd t.
\end{equation*}
\end{example}

\begin{figure}[!ht]\centering
  \begin{tikzpicture}[scale=.8]
    \fill[red,opacity=.3] (2,0) arc (0:360:1cm);
    \fill[red,opacity=.3] (0,0) arc (0:360:1cm);
    \fill[red,opacity=.3] (0,0) arc (180:90:1cm) -- (1,1) arc (0:180:1cm) -- (-1,1) arc (90:0:1cm);
    \fill[red,opacity=.3] (0,0) arc (180:270:1cm) -- (1,-1) arc (0:-180:1cm) -- (-1,-1) arc (-90:0:1cm);
    
    \fill[blue,opacity=.3] (0,0) arc (-90:0:1cm) -- (1,1) arc (90:180:1cm);
    \fill[blue,opacity=.3] (0,0) arc (0:90:1cm) -- (-1,1) arc (180:270:1cm);
    \fill[blue,opacity=.3] (0,0) arc (90:180:1cm) -- (-1,-1) arc (-90:0:1cm);
    \fill[blue,opacity=.3] (0,0) arc (180:270:1cm) -- (1,-1) arc (0:90:1cm);

    \node at (.3,1.4) {1};
    \node at (.5,.5) {2};
    
    \draw[->] (0,-2.3) -- (0,2.3) node[anchor = east]{$y_2$};
    \draw[->] (-3,0) -- (3,0) node[anchor = north]{$y_1$};
    
    \draw (2,-0.08) -- (2,0.08);
    \node[anchor = north] at (2,0){$2k_0$};
    \draw (-2,-0.08) -- (-2,0.08);
    \node[anchor = north] at (-2,0){$-2k_0$};
  \end{tikzpicture}
  \qquad
  \includegraphics[trim={4cm 2cm 3.5cm 1.8cm},clip,width=0.25\textwidth]{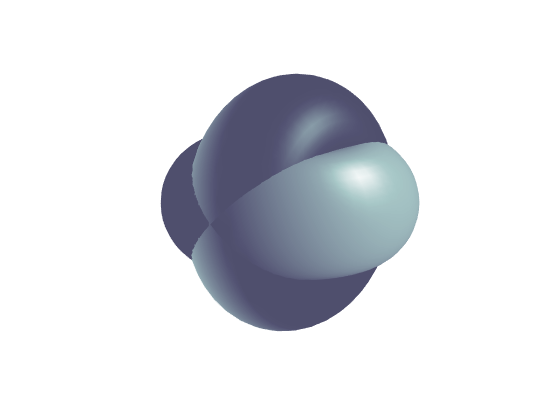}
  \caption{\emph{Left:} 2D Fourier coverage for \autoref{ex:angle+rot}. In the purple area, the Banach indicatrix is 2, in the red area it is 1. \emph{Right:} 3D Fourier coverage for \autoref{ex:3d}.}
  \label{fig:angle+rot}
\end{figure}

\begin{example}[Object rotation in 3D] \label{ex:3d}
We consider an experiment similar to \autoref{ex:2d} but in $\R^3$. The object rotates around the $r_1$-axis, the wave number $k_0$ is fixed and $\rM \in I^+$. The incidence direction $\bs = (0,0,1)^\top$ leads to Devaney's filtered backpropagation formula \cite{Dev82}. An illustration of the Fourier coverage for this setup can be found in \cite[Fig.\ 3]{KirQueRitSchSet21}.

As in $\R^2$, choosing $\bs = (0,1,0)^\top$, i.e.\ parallel to the measurement plane, yields a larger coverage, cf.\ \autoref{fig:cover3}. In contrast to the 2D setting, however, the Fourier coverage is considerably smaller than the maximal one. In particular, it suffers from the missing cone problem. In this case the missing regions around the origin can be filled, for instance, by subsequently rotating the object around the $r_2$-axis while illuminating in direction $\bs = (1,0,0)^\top$. The resulting coverage $\mathcal{Y}$ is a union of two solid horn tori, one radially symmetric about the $r_1$-axis and the other radially symmetric about the $r_2$-axis, see \autoref{fig:angle+rot} right. The filtered backpropagation formula reads
\begin{equation*}
	f_{\mathcal{Y}}(\br) = -\frac{\i}{2\pi^2 k_0} \int_{0}^{4 \pi} \int_{\mathcal{B}_{k_0}} \frac{\kappa \e^{\i T(\bx,t)\cdot \br - \i \kappa \rM} \tilde\ktran u_t(\bx,\rM)}{\operatorname{Card}\left( T^{-1}(T(\bx,t)) \right)} \dd \bx \dd t
\qquad\text{for all } \br \in \mathbb{R}^3,
\end{equation*}
where $T$ is defined according to \eqref{eq:T} with $\bs(t) = (0,1,0)^\top$ for $t \in [0,2\pi]$ and $\bs(t) = (1,0,0)^\top$ for $t \in (2\pi,4\pi]$ and the rotation matrix
\begin{align*}
	R(t) =
	\begin{dcases}
		\begin{pmatrix}	1 & 0 & 0 \\ 0 & \phantom{-}\cos t & \sin t \\ 0 & -\sin t & \cos t \end{pmatrix}, & t \in [0,2\pi], \\
		\begin{pmatrix}	\phantom{-}\cos t & 0 & \sin t \\ 0 & 1 & 0 \\ -\sin t & 0 & \cos t \end{pmatrix}, & t \in (2\pi,4\pi].
	\end{dcases}
\end{align*}
The Banach indicatrix $\operatorname{Card}\left( T^{-1}(\by)\right)$ equals $2$ if $\by$ lies in the overlap of the two solid tori, and equals $1$ otherwise. For reasons of symmetry, half a rotation of the object about each axis is actually enough to compute $f_{\mathcal{Y}}$ if $f$ is real-valued, similar to \autoref{ex:2d}.
\end{example}

\section{Numerics}

\subsection{Discretization}\label{sec:discretization}

For discretizing the filtered backpropagation formulae of Sections \ref{sec:gen-backprop} and \ref{sec:non-absorbing}, we extend the approach of \cite{KirQueRitSchSet21} to our general setting with some modifications for the Banach indicatrix.
We consider the time steps $t_n \coloneqq nL/N$ for $n=1,\dots, N$,
and quadrature points $\bx_{m} \in \B_{1}^{d-1}$ for $m =  1,\dots, M$ that lie on a uniform grid.
From \eqref{eq:reconstruction},
we obtain the \emph{discrete backpropagation}
\begin{equation}\label{eq:reconstruction_dis}
  f_{\mathcal{Y}}(\br) 
  \approx 
  {(2\pi)^{-\frac{1+d}{2}}}
  \frac{\lvert{\B^{d-1}_{k_0}}\rvert\, L}{MN}
  \sum_{m=1}^{M}
  \sum_{n=1}^{N}
  \frac{2 \kappa(\bz_{m,n}) \e^{\i T(\bz_{m,n}) \cdot \br}\, \tilde\ktran u_t(k_0\bx_m,\rM) \abs{\det \left(\nabla T(\bz_{m,n})\right)} }{k_0^2(t_n)\, \i \e^{\i \kappa(\bz_{m,n}) \rM} \operatorname{Card}(T^{-1}(T(\bz_{m,n})))},
\end{equation}
where $\bz_{m,n} \coloneqq (k_0(t_n) \bx_m, t_n)$.
For a non-absorbing object as of \autoref{thm:reconstruction2},
we approximate $f_{\mathcal{Y}_\sym}$ analogously to \eqref{eq:reconstruction_dis}, where we replace $\operatorname{Card}(T^{-1}(\cdot))$ by $\operatorname{Card}(T_\sym^{-1}(\cdot))$ 
and take twice the real part of the sum.
We evaluate $f_{\mathcal{Y}}$ on a uniform grid 
\begin{equation} \label{eq:x-grid}
\br_\bp = 2 \rM \bp,
\qquad 
\bp \in \mathcal{I}_P^d \coloneqq \{ -\tfrac P2,\,\dots,\,\tfrac P2-1 \}^d,
\end{equation}
for $P\in\N$.
The \emph{nonuniform discrete Fourier transform} (NDFT) $\mathbf A\colon \C^{P^d}\to\C^{J}$
of a vector $\mathbf f \in \C^{P^d}$ at points $\by_j\in\R^d$, $j=1,\dots,J$, 
and its adjoint $\mathbf A^*\colon \C^{J}\to\C^{P^d}$ of $\ba\in\C^j$
are defined by 
\[
(\mathbf A\mathbf f)_j
\coloneqq
\sum_{\bp \in \mathcal{I}_P^d} 
\mathbf f_{\bp}\, \e^{\i \by_{j} \cdot \bp},
\qquad 
(\mathbf A^*\ba)_\bp
\coloneqq
\sum_{j=1}^J 
\mathbf a_{j}\, \e^{\i \by_{j} \cdot \bp}.
\]
With appropriate scaling and the enumeration $\by_{j(m,n)} = T(\bz_{m,n})$, the evaluation of \eqref{eq:reconstruction_dis} 
corresponds to an adjoint NDFT,
which can be computed efficiently in $\mathcal{O}(P^d \log P + NM)$ arithmetic operations, see \cite[Chap.\ 7]{PlPoStTa18}.
The Jacobian determinant $\abs{\det(\nabla T)}$, see \eqref{eq:jacobian},
can be approximated using finite differences.

\paragraph{Banach indicatrix}
The only part of \eqref{eq:reconstruction_dis} that is, in general, hard to determine analytically is the Banach indicatrix $\operatorname{Card}(T^{-1}(\by))$, 
which we approximate as follows.
For simplicity, we only look at the case of continuous parameters,
but we may apply the procedure for finitely many subintervals of $t$.
The indicatrix counts how often a point $\by\in\B^d_{2k_0}$ is ``hit'' by the transformation~$T$.
In the discrete setting, however, it is unlikely that a point $\by$ is exactly hit by $T(\bz_{m,n})$ for any $m,n$.
By \eqref{eq:T}, we can express the coverage for fixed time $t$ as the hemisphere 
\begin{equation} \label{eq:Yt}
\left\{T(\bx,t) : \bx\in\B^{d-1}_{k_0}\right\}
=
\left\{\by\in\R^d :
\abs{\by + k_0(t) R(t) \bs(t)} = k_0(t),\,
\by\cdot R(t) \be^d > -k_0(t) \bs(t)\cdot \be^d
\right\},
\end{equation}
which moves continuously with $t$. For sufficiently close time steps, a point $\by$ is hit by $T$ between the time steps $t_{n-1}$ and~$t_n$
if the sign of 
$
\abs{\by + k_0(t) R(t) \bs(t)} - k_0(t) 
$
changes between these time steps.
Hence, we approximate $\operatorname{Card}(T^{-1}(\by))$ by
\begin{equation} \label{eq:indicatrix}
\sum_{n=1}^N
\frac{s(n)}{2} \abs{\sgn\left(\abs{\by + k_0(t_{n-1}) R({t_{n-1}}) \bs(t_{n-1})} - k_0(t_{n-1}) \right)
  - \sgn\left(\abs{\by + k_0(t_{n}) R({t_{n}}) \bs(t_{n})} - k_0(t_{n}) \right)},
\end{equation}
where the sign function is given in \eqref{eq:sgn} and
$$
s(n) \coloneqq
\begin{cases}
1
,& \text{if } \by \cdot (R(t)\be^d) > -k_0(t_n) \bs(t_n)\cdot \be^d,\\
0 ,& \text{otherwise}.
\end{cases}
$$
Here the factor $1/2$ compensates the fact that a full sign change of the argument changes the $\sgn(\cdot)$ function by~2.

\paragraph{Inverse NDFT and density compensation}

We compare the discrete backpropagation with other approaches.
The forward model \eqref{eq:recon}, which maps $f$ to $\tilde\ktran u_t(\cdot, \rM)$,
can be discretized via an NDFT:
with the equispaced grid $\br_\bp$ from \eqref{eq:x-grid},
we have
\begin{equation} \label{eq:forward-discrete}
  \tilde \ktran u_{t_n}(\bx_m, \rM)
  \approx 
  \left(\frac{2\rM}{P}\right)^d
  \sqrt{\frac\pi2}\, \frac{ \i \e^{\i \kappa \rM}}{\kappa}
  \, {k_0^2}\, 
  \sum_{\bp\in \mathcal{I}_P^d}
  \e^{-\i T (\bz_{m,n}) \cdot \br_\bp}
  f \left(\br_{\bp} \right).
\end{equation}
The \emph{inverse NDFT} method \cite{KirQueRitSchSet21} consists in applying a conjugate gradient (CG) method to solve $\mathbf A \mathbf f = \bg$, where 
$$
\bg = 
\left( \tilde\ktran u_{t_n}(\bx_m, \rM) \left(\frac{P}{2\rM}\right)^d \frac{-\i \sqrt2\, \kappa \e^{-\i\kappa\rM}}{\sqrt{\pi}\, k_0^2}
 \right)_{m,n=1}^{M,N}
 $$ 
consists of the Fourier-transformed measurements, $\mathbf f = (f(\br_\bp))_{\bp\in \mathcal I_P^d}$, and $\mathbf A$ is the NDFT.
Note that our implementation of the inverse NDFT enforces $f$ to be real-valued as described in \cite[sect.\ 5.2]{BeiQue22}.

There are different approaches for numerical inversion of the NDFT, see \cite{AdcGatHan14,GelSon14} and \cite[sect.\ 3]{KirPot23}.
Furthermore,
we consider the adjoint NDFT with \emph{density compensation} factors that can be computed from $\by_{m,n}$ via a conjugate gradient (CG) method, see \cite{KirPot23a}. 
These factors play the same role as the weights in the backpropagation formula \eqref{eq:reconstruction_dis}
because they only depend on the measurement setup, i.e.\ the transformation $T$,
but not on the measured data $u_t$, 
and can therefore be precomputed.

\subsection{Numerical tests}\label{sec:numerics}

We consider a two-dimensional, real-valued test function $f\in L^1(\R^2)$ that contains both convex and nonconvex shapes, see \autoref{fig:heart1-f}.
Note that, different from \cite{KirQueRitSchSet21}, we plot the normalized scattering potential $f$, see \eqref{eq:f}.
We discretize $f$ on a $144\times 144$ grid and take a fixed wave number $k_0 = 2\pi$ and the measurement line $x_2=\rM=20$.
Both for the inverse NDFT and the density compensation, we perform 20~CG iterations, the same as in \cite{KirQueRitSchSet21}.
We use the library \cite{nfft3,KeKuPo09} for the (adjoint) NDFT in all tested algorithms.
As our main goal is to examine the different backpropagation formulae, we generate the sinogram data $u_t(x,\rM)$ with the same forward model~\eqref{eq:forward-discrete}.
In practice, one measures the total field $\utot_t(x,\rM) \coloneqq u_t(x,\rM)+\e^{\i k_0 \bs(t)\cdot(x,\rM)}$, i.e., the sum of the scattered field and incident field.

\tikzset{font=\tiny}
\newcommand{\modelwidth} {3.67cm}
\pgfplotsset{
  colormap={parula}{
    rgb255=(53,42,135)
    rgb255=(15,92,221)
    rgb255=(18,125,216)
    rgb255=(7,156,207)
    rgb255=(21,177,180)
    rgb255=(89,189,140)
    rgb255=(165,190,107)
    rgb255=(225,185,82)
    rgb255=(252,206,46)
    rgb255=(249,251,14)}}
\pgfmathsetmacro{\xmin}{-16.970563}
\pgfmathsetmacro{\xmax}{16.734860}
\pgfmathsetmacro{\cmin} {-0.001267} \pgfmathsetmacro{\cmax} {0.013932}

\begin{figure}[ht]\centering
\begin{subfigure}[t]{.33\textwidth}\centering
  \begin{tikzpicture}
    \begin{axis}[
      width=\modelwidth,
      height=\modelwidth,
      enlargelimits=false,
      scale only axis,
      axis on top,
      point meta min=\cmin,point meta max=\cmax,
      colorbar,colorbar style={
        width=.15cm, xshift=-0.5em,  
      },
      xlabel={$r_1$}, ylabel={$r_2$},
      x label style={at={(axis description cs:0.5,-0.1)},anchor=north},
      y label style={at={(axis description cs:-0.1,.5)},rotate=0,anchor=south},
      ]
      \addplot graphics [
      xmin=\xmin, xmax=\xmax,  ymin=\xmin, ymax=\xmax,
      ] {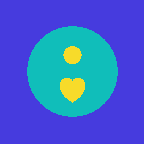};
    \end{axis}
  \end{tikzpicture}%
  \caption{Function $f$}
\end{subfigure}\hfill
\begin{subfigure}[t]{.33\textwidth}\centering
  \begin{tikzpicture}
  \begin{axis}[
    width=\modelwidth,
    height=\modelwidth,
    enlargelimits=false,
    scale only axis,
    axis on top,
    point meta min=0.9320, point meta max=1.2291,
    colorbar,colorbar style={
      width=.15cm, xshift=-0.5em, 
    },
    xlabel={$t$}, ylabel={$x$},
    x label style={at={(axis description cs:0.5,-0.1)},anchor=north},
    y label style={at={(axis description cs:-0.1,.5)},rotate=0,anchor=south},
    ]
    \addplot graphics [
    xmin=0, xmax=2.4,  ymin=64, ymax=-64,
    ] {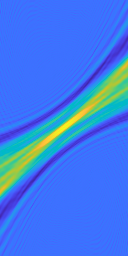};
  \end{axis}
  \end{tikzpicture}%
  \caption{Sinogram for \autoref{fig:heart1-angle} (angle scan) \label{fig:heart1-angle-sinogram}}
\end{subfigure}\hfill
\begin{subfigure}[t]{.33\textwidth}\centering
  \begin{tikzpicture}
  \begin{axis}[
    width=\modelwidth,
    height=\modelwidth,
    enlargelimits=false,
    scale only axis,
    axis on top,
    point meta min=0.9320, point meta max=1.2291,
    colorbar,colorbar style={
      width=.15cm, xshift=-0.5em, 
    },
    xlabel={$t$}, ylabel={$x$},
    x label style={at={(axis description cs:0.5,-0.1)},anchor=north},
    y label style={at={(axis description cs:-0.1,.5)},rotate=0,anchor=south},
    ]
    \addplot graphics [
    xmin=0, xmax=4.8,  ymin=-64, ymax=64,
    ] {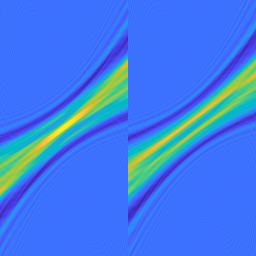};
  \end{axis}
  \end{tikzpicture}
  \caption{Sinogram for \autoref{fig:heart1-anglerot} (two angle scans combined) \label{fig:heart1-anglerot-sinogram}}
  \end{subfigure}
  \caption{Ground truth $f$ (left) and absolute value of the sinograms $\abs{\utot_t(x, \rM)}$. \label{fig:heart1-f}}
\end{figure}

\paragraph*{Angle scan}
We first consider an angle scanning setup, see \autoref{sec:incidence}, with a fixed position of the object, $R(t) = \mathrm{id}$, 
and the incidence direction
$\bs(t) = \bs_0(t)/\abs{\bs_0(t)}$,
where $\bs_0(t) = (t-1.2,1)$
for $t\in[0,2.4]$ with $N=128$ time steps.
We discretize $x$ on the equispaced grid $2M^{-1} \mathcal{I}_M^1$ with $M=128$.
\autoref{fig:heart1-angle-sinogram} depicts the simulated sinogram $\utot_t(x)$.
We always take the real part of the reconstructions of $f$ and compare the quality using the peak signal-to-noise ratio (PSNR) and structural similarity index measure (SSIM).

\autoref{fig:heart1-angle} shows the reconstructions, the Fourier coverage $\mathcal Y$ and $\mathcal Y_\sym$, and the respective Banach indicatrix $\operatorname{Card}( T^{-1}(\by))$ or $\operatorname{Card}( T_\sym^{-1}(\by))$ estimated via \eqref{eq:indicatrix}.
Both the backpropagation \eqref{eq:reconstruction_dis} and its symmetrized version yield similar results: they reveal the correct structure of the object, but underestimate the values of~$f$.
Moreover, they produce some artifacts due to the missing parts in the Fourier coverage,
also known as the ``missing cone'', cf.\ \cite{krauze2020optical}.
The inverse NDFT gives a better reconstruction with the correct level of the peaks of $f$, but still shows some missing cone artifacts shaped similarly as in the backpropagation. The density compensation, which yields the largest error, overestimates the function and produces smoother artifacts that makes the two inclusions almost seem like a single large one.

\renewcommand{\modelwidth}{4.2cm}

\begin{figure}[ht]\centering
  \begin{subfigure}[t]{.325\textwidth}\centering
    \begin{tikzpicture}
    \begin{axis}[
      width=\modelwidth, height=\modelwidth,
      enlargelimits=false,
      scale only axis,
      axis on top,
      ticks=none,
      point meta min=\cmin,point meta max=\cmax,
      colorbar,colorbar style={
        width=.15cm, xshift=-0.5em, 
      },
      ]
      \addplot graphics [
      xmin=\xmin, xmax=\xmax,  ymin=\xmin, ymax=\xmax,
      ] {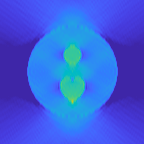};
    \end{axis}
    \end{tikzpicture}
    \caption{Backpropagation $f_{\mathcal{Y}}$ from \eqref{eq:reconstruction_dis},\\ PSNR 20.7, SSIM 0.343 }
  \end{subfigure}\hfill
  \begin{subfigure}[t]{.325\textwidth}\centering
  \begin{tikzpicture}
    \begin{axis}[
      width=\modelwidth,
      height=\modelwidth,
      enlargelimits=false,
      scale only axis,
      axis on top,
      ticks=none,
      point meta min=.8, point meta max=4.2,
      colorbar,colorbar style={
        width=.15cm, xshift=-0.5em,   
      },
      ]
      \addplot graphics [
      xmin=-12, xmax=12,  ymin=-12, ymax=12,
      ] {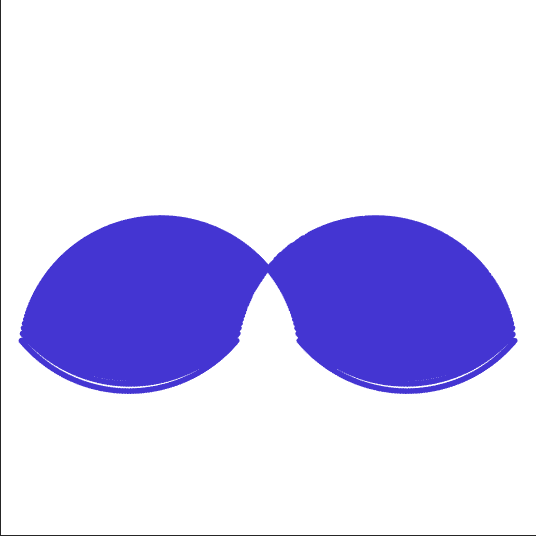};
    \end{axis}
  \end{tikzpicture}
  \caption{Fourier coverage $\mathcal Y$ for (a), the color is the Banach indicatrix $\operatorname{Card}( T^{-1}(\by))$}
  \end{subfigure}\hfill
  \begin{subfigure}[t]{.325\textwidth}\centering
  \begin{tikzpicture}
    \begin{axis}[
      width=\modelwidth,
      height=\modelwidth,
      enlargelimits=false,
      scale only axis,
      axis on top,
      ticks=none,
      point meta min=\cmin,point meta max=\cmax,
      colorbar,colorbar style={
        width=.15cm, xshift=-0.5em, 
      },
      ]
      \addplot graphics [
      xmin=\xmin, xmax=\xmax,  ymin=\xmin, ymax=\xmax,
      ] {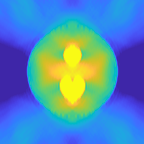};
    \end{axis}
  \end{tikzpicture}
  \caption{Density compensation \cite{KirPot23a},\\ PSNR 17.9, SSIM 0.273}
  \end{subfigure}
  
  \begin{subfigure}[t]{.325\textwidth}\centering
  \begin{tikzpicture}
    \begin{axis}[
      width=\modelwidth,
      height=\modelwidth,
      enlargelimits=false,
      scale only axis,
      axis on top,
      ticks=none,
      point meta min=\cmin,point meta max=\cmax,
      colorbar,colorbar style={
        width=.15cm, xshift=-0.5em, 
      },
      ]
      \addplot graphics [
      xmin=\xmin, xmax=\xmax,  ymin=\xmin, ymax=\xmax,
      ] {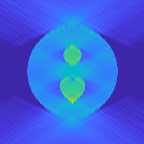};
    \end{axis}
  \end{tikzpicture}
  \caption{Backpropagation $f_{\mathcal{Y}_\sym}$ with symmetrization, PSNR 20.9, SSIM 0.323}
  \end{subfigure}\hfill
  \begin{subfigure}[t]{.325\textwidth}\centering
  \begin{tikzpicture}
    \begin{axis}[
      width=\modelwidth,
      height=\modelwidth,
      enlargelimits=false,
      scale only axis,
      axis on top,
      ticks=none,
      point meta min=.8, point meta max=4.2, 
      colorbar,colorbar style={
        width=.15cm, xshift=-0.5em,  
      },
      ]
      \addplot graphics [
      xmin=-12, xmax=12,  ymin=-12, ymax=12,
      ] {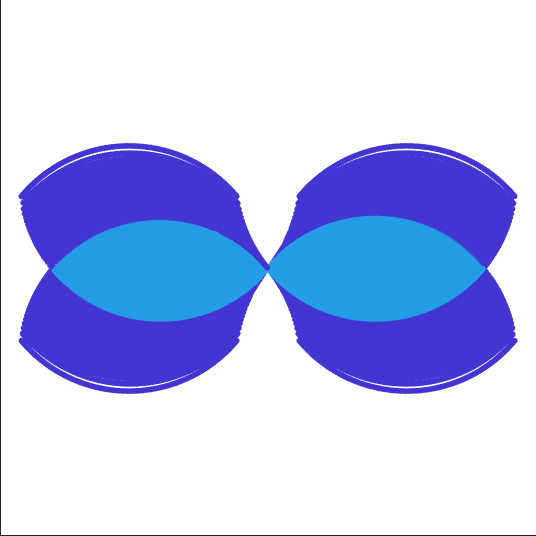};
    \end{axis}
  \end{tikzpicture}
  \caption{Banach indicatrix $\operatorname{Card}(T_\sym^{-1}(\by))$ and Fourier coverage $\mathcal{Y}_\sym$ for (d) \label{fig:heart1-angle-crofton-sym}}
  \end{subfigure}\hfill
  \begin{subfigure}[t]{.325\textwidth}\centering
  \begin{tikzpicture}
    \begin{axis}[
      width=\modelwidth,
      height=\modelwidth,
      enlargelimits=false,
      scale only axis,
      axis on top,
      ticks=none,
      point meta min=\cmin,point meta max=\cmax,
      colorbar,colorbar style={
        width=.15cm, xshift=-0.5em,  
      },
      ]
      \addplot graphics [
      xmin=\xmin, xmax=\xmax,  ymin=\xmin, ymax=\xmax,
      ] {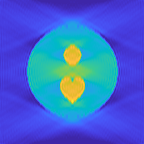};
    \end{axis}
  \end{tikzpicture}
  \caption{Inverse NDFT \cite{KirQueRitSchSet21},\\ PSNR 28.0, SSIM 0.593}
  \end{subfigure}
  
  \caption{Reconstructions for angle scan with fixed object.
    For comparison, we also show the inverse NDFT \cite{KirQueRitSchSet21} and the density compensation \cite{KirPot23a}.
  \label{fig:heart1-angle}}
\end{figure}

\paragraph*{Angle scan and 90° rotation}
Our second setup demonstrates the necessity of non-smooth parameters.
We take the above experiment,
repeat it with the object rotated by 90°, and combine the data of both parts.
Formally, we set 
$\bs(t) = \bs_0(t)/\abs{\bs_0(t)}$
where $\bs_0(t) = (\alpha(t),1)$ with
\begin{alignat*}{6}
  \alpha(t) &= t-1.2 &&\quad\text{and}\quad && 
  R(t) = \begin{psmallmatrix} 1 & 0 \\ 0 & 1 \end{psmallmatrix} \quad&& \text{if } t\in[0,2.4),\\ 
  \alpha(t) &= t-3.6 &&\quad\text{and}\quad &&
  R(t) = \begin{psmallmatrix} 0 & 1 \\ -1 & 0 \end{psmallmatrix}\quad &&
  \text{if } t\in[2.4,4.8].
\end{alignat*}
The sinogram $\utot_t(x,\rM)$ in \autoref{fig:heart1-anglerot-sinogram} shows the discontinuity at $t=2.4$. The reconstructions are plotted in \autoref{fig:heart1-anglerot}.
Again, the backpropagation yields better results than the density compensation. We see that the symmetrized backpropagation in \autoref{fig:heart1-anglerot-bp-sym} gives a slightly better reconstruction than the one without in \autoref{fig:heart1-anglerot-bp} and is almost comparable with the inverse NDFT in \autoref{fig:heart1-anglerot-cg}.
Furthermore, \autoref{fig:heart1-anglerot-bp-nocrofton} indicates that the backpropagation becomes considerably worse without the Banach indicatrix.

\begin{figure}[ht!]\centering  
  \begin{subfigure}[t]{.325\textwidth}\centering
    \begin{tikzpicture}
      \begin{axis}[
        width=\modelwidth,
        height=\modelwidth,
        enlargelimits=false,
        scale only axis,
        axis on top,
        ticks=none,
        point meta min=\cmin,point meta max=\cmax,
        colorbar,colorbar style={
          width=.15cm, xshift=-0.5em,  
        },
        ]
        \addplot graphics [
        xmin=\xmin, xmax=\xmax,  ymin=\xmin, ymax=\xmax,
        ] {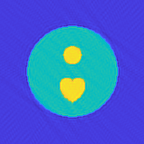};
      \end{axis}
    \end{tikzpicture}
    \caption{Backpropagation $f_{\mathcal{Y}}$ from \eqref{eq:reconstruction_dis},\\ PSNR 37.5, SSIM 0.958 \label{fig:heart1-anglerot-bp}}
  \end{subfigure}\hfill
  \begin{subfigure}[t]{.325\textwidth}\centering
  \begin{tikzpicture}
    \begin{axis}[
      width=\modelwidth,
      height=\modelwidth,
      enlargelimits=false,
      scale only axis,
      axis on top,
      ticks=none,
      point meta min=.8, point meta max=4.2, 
      colorbar,colorbar style={
        width=.15cm, xshift=-0.5em,  
      },
      ]
      \addplot graphics [
      xmin=-12, xmax=12,  ymin=-12, ymax=12,
      ] {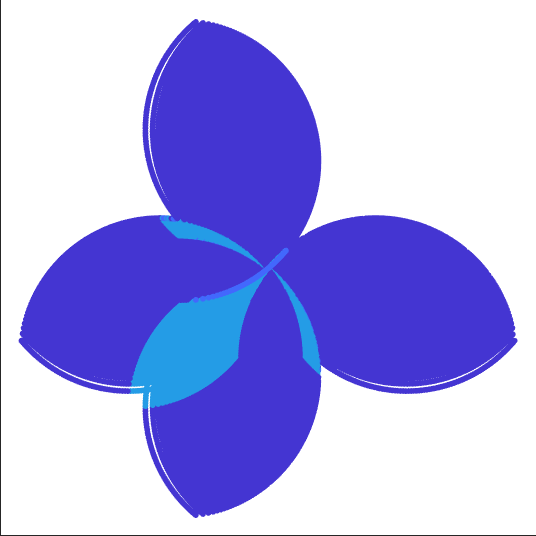};
    \end{axis}
  \end{tikzpicture}
  \caption{Fourier coverage $\mathcal Y$ for (a), the color is the Banach indicatrix $\operatorname{Card}(T^{-1}(\by))$}
  \end{subfigure}\hfill
  \begin{subfigure}[t]{.325\textwidth}\centering
  \begin{tikzpicture}
    \begin{axis}[
      width=\modelwidth,
      height=\modelwidth,
      enlargelimits=false,
      scale only axis,
      axis on top,
      ticks=none,
      point meta min=\cmin,point meta max=\cmax,
      colorbar,colorbar style={
        width=.15cm, xshift=-0.5em,  
      },
      ]
      \addplot graphics [
      xmin=\xmin, xmax=\xmax,  ymin=\xmin, ymax=\xmax,
      ] {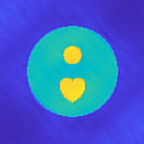};
    \end{axis}
  \end{tikzpicture}
  \caption{Density compensation \cite{KirPot23a},\\ PSNR 33.8, SSIM 0.624}
  \end{subfigure}
  
  \begin{subfigure}[t]{.325\textwidth}\centering
    \begin{tikzpicture}
      \begin{axis}[
        width=\modelwidth,
        height=\modelwidth,
        enlargelimits=false,
        scale only axis,
        axis on top,
        ticks=none,
        point meta min=\cmin,point meta max=\cmax,
        colorbar,colorbar style={
          width=.15cm, xshift=-0.5em,  
        },
        ]
        \addplot graphics [
        xmin=\xmin, xmax=\xmax,  ymin=\xmin, ymax=\xmax,
        ] {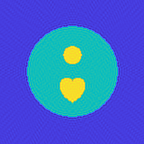};
      \end{axis}
    \end{tikzpicture}
    \caption{Symmetrized Backpropagation $f_{\mathcal{Y}_\sym}$ from \eqref{eq:reconstruction2}, PSNR 38.3, SSIM 0.958 \label{fig:heart1-anglerot-bp-sym}}
  \end{subfigure}\hfill
  \begin{subfigure}[t]{.325\textwidth}\centering
  \begin{tikzpicture}
    \begin{axis}[
      width=\modelwidth,
      height=\modelwidth,
      enlargelimits=false,
      scale only axis,
      axis on top,
      ticks=none,
      point meta min=.8, point meta max=4.2, 
      colorbar,colorbar style={
        width=.15cm, xshift=-0.5em,  
      },
      ]
      \addplot graphics [
      xmin=-12, xmax=12,  ymin=-12, ymax=12,
      ] {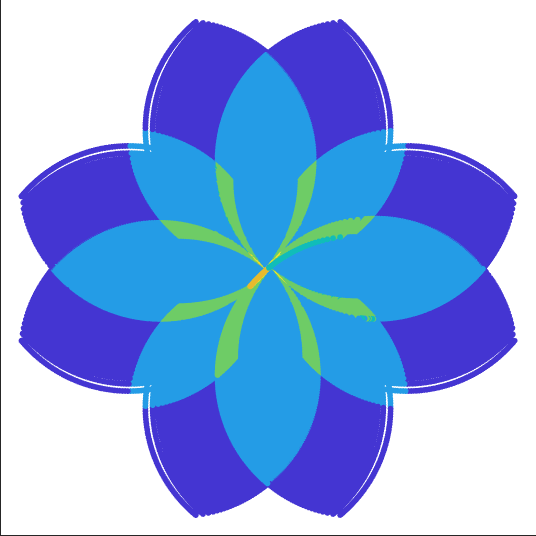};
    \end{axis}
  \end{tikzpicture}
  \caption{Banach indicatrix $\operatorname{Card}(T_\sym^{-1}(\by))$ and Fourier coverage $\mathcal{Y}_\sym$ for (d)}
  \end{subfigure}\hfill
  \begin{subfigure}[t]{.325\textwidth}\centering
    \begin{tikzpicture}
    \begin{axis}[
      width=\modelwidth,
      height=\modelwidth,
      enlargelimits=false,
      scale only axis,
      axis on top,
      ticks=none,
      point meta min=\cmin,point meta max=\cmax,
      colorbar,colorbar style={
        width=.15cm, xshift=-0.5em,  
      },
      ]
      \addplot graphics [
      xmin=\xmin, xmax=\xmax,  ymin=\xmin, ymax=\xmax,
      ] {images/heart1-anglerot-rec-bp-sym.png};
    \end{axis}
    \end{tikzpicture}
    \caption{Inverse NDFT \cite{KirQueRitSchSet21},\\ PSNR 38.5, SSIM 0.980 \label{fig:heart1-anglerot-cg}}
  \end{subfigure}
  \caption{Two angle scans combined: the first scan for the initial object, the second scan with the object rotated by 90°.
  \label{fig:heart1-anglerot}}
\end{figure}
\begin{figure}[ht!] \centering
  \begin{tikzpicture}
    \begin{axis}[
      width=\modelwidth,
      height=\modelwidth,
      enlargelimits=false,
      scale only axis,
      axis on top,
      ticks=none,
      point meta min=\cmin,point meta max=\cmax,
      colorbar,colorbar style={
        width=.15cm, xshift=-0.5em,  
      },
      ]
      \addplot graphics [
      xmin=\xmin, xmax=\xmax,  ymin=\xmin, ymax=\xmax,
      ] {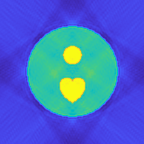};
    \end{axis}
  \end{tikzpicture}
  \caption{Backpropagation $f_{\mathcal{Y}}$ for constant indicatrix, i.e., we use \eqref{eq:reconstruction} with $\operatorname{Card}(T_\sym^{-1}(\cdot)) \equiv 1$, otherwise same setup as in \autoref{fig:heart1-anglerot}. PSNR 31.2, SSIM 0.605 \label{fig:heart1-anglerot-bp-nocrofton}}
\end{figure}

To evaluate the stability of the algorithms, we repeat the above simulation and add Gaussian white noise to the measured sinogram of the total field~$\utot_t$ with a relative noise level of 1\,\%, i.e., $\norm{\utot_\mathrm{noisy}-\utot}/\norm{\utot}\approx1\,\%$ with the discrete $L^2$ norm $\norm{\utot}=(\sum_{n,m=1}^{N,M}|\utot_{t_n}(x_m,\rM)|^2)^{1/2}$.
Note that the noise level relative to the scattered field $u=\utot-\ui$ is larger, namely $\norm{u_\mathrm{noisy}-u}/\norm{u}\approx 5.8\,\%$.
For the reconstructions in \autoref{fig:heart1-anglerot-noise}\,(a)--(d), we did not apply further regularization.
All methods yield visually similar results: they correctly reconstruct the right shape and amplitude, but are distorted by some noise artifacts distributed almost equally over the image.
Only the density compensation has structural artifacts visible in the background.
The error metrics are best for the backpropagation~$f_{\mathcal{Y}}$ by a small margin.
As in \cite{BeiQue22}, we post-process the noisy reconstructions via total variation (TV) denoising and enforcing the non-negativity of the reconstruction.
In particular, we employ primal-dual splitting \cite{ChaCasCreNovPoc10} to minimize 
\begin{equation}\label{eq:tv}
f\mapsto \tfrac12\norm{f-f_{\mathcal Y}}_{L^2}^2+\lambda\norm{f}_{\mathrm{TV}}+\chi_{\ge0}(f),
\end{equation}
where $f_{\mathcal Y}$ is the backpropagation, $\lambda>0$ a regularization parameter, $\norm{\cdot}_{\mathrm{TV}}$ the total variation seminorm, and $\chi_{\ge0}(f)=0$ if $f\ge0$ everywhere and $+\infty$ otherwise.
The denoised images shown in \autoref{fig:heart1-anglerot-noise}\,(e)--(h) are significantly better. Again, all methods perform on a similar level.
Noteworthy, the symmetrized backpropagation $f_{\mathcal{Y}_\mathrm{sym}}$ profits the most from denoising, even slightly beating the denoised inverse NDFT.

\begin{figure}[htb!] 
  \renewcommand{\modelwidth}{3.56cm}
  \begin{subfigure}[t]{.235\textwidth}
    \begin{tikzpicture}
      \begin{axis}[
        width=\modelwidth,
        height=\modelwidth,
        enlargelimits=false,
        scale only axis,
        axis on top,
        ticks=none,
        point meta min=\cmin,point meta max=\cmax,
        ]
        \addplot graphics [
        xmin=\xmin, xmax=\xmax,  ymin=\xmin, ymax=\xmax,
        ] {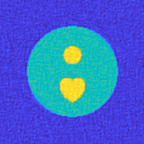};
      \end{axis}
    \end{tikzpicture}
    \caption{Backpropagation $f_{\mathcal{Y}}$,\\ PSNR 34.8, SSIM 0.800}
  \end{subfigure}
  \begin{subfigure}[t]{.235\textwidth}
    \begin{tikzpicture}
      \begin{axis}[
        width=\modelwidth,
        height=\modelwidth,
        enlargelimits=false,
        scale only axis,
        axis on top,
        ticks=none,
        point meta min=\cmin,point meta max=\cmax,
        ]
        \addplot graphics [
        xmin=\xmin, xmax=\xmax,  ymin=\xmin, ymax=\xmax,
        ] {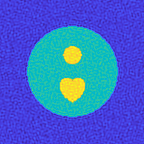};
      \end{axis}
    \end{tikzpicture}
    \caption{Symmetrized Backpropagation $f_{\mathcal{Y}_\sym}$,\\ PSNR 33.6, SSIM 0.722}
  \end{subfigure}
  \begin{subfigure}[t]{.235\textwidth}
  \begin{tikzpicture}
    \begin{axis}[
      width=\modelwidth,
      height=\modelwidth,
      enlargelimits=false,
      scale only axis,
      axis on top,
      ticks=none,
      point meta min=\cmin,point meta max=\cmax,
      ]
      \addplot graphics [
      xmin=\xmin, xmax=\xmax,  ymin=\xmin, ymax=\xmax,
      ] {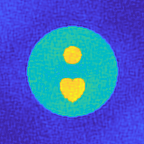};
    \end{axis}
  \end{tikzpicture}
  \caption{Density compensation \cite{KirPot23a},\\ PSNR 32.4, SSIM 0.553}
  \end{subfigure}
  \begin{subfigure}[t]{.275\textwidth}
    \begin{tikzpicture}
      \begin{axis}[
        width=\modelwidth,
        height=\modelwidth,
        enlargelimits=false,
        scale only axis,
        axis on top,
        ticks=none,
        point meta min=\cmin,point meta max=\cmax,
        ]
        \addplot graphics [
        xmin=\xmin, xmax=\xmax,  ymin=\xmin, ymax=\xmax,
        ] {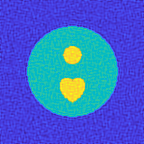};
      \end{axis}
    \end{tikzpicture}\hspace*{-5.5mm}%
    \begin{tikzpicture}
      \pgfplotscolorbardrawstandalone[colormap name=parula,
      point meta min=\cmin,point meta max=\cmax, 
      colorbar style={height=\modelwidth, width=.15cm, xshift=0mm, 
      },
      yticklabel style={
        xshift=-.5mm,yshift=0mm,font=\tiny
      } ]
    \end{tikzpicture}
    \caption{Inverse NDFT \cite{KirQueRitSchSet21},\\ PSNR 33.9, SSIM 0.743}
  \end{subfigure}
  
  \begin{subfigure}[t]{.235\textwidth}
    \begin{tikzpicture}
      \begin{axis}[
        width=\modelwidth,
        height=\modelwidth,
        enlargelimits=false,
        scale only axis,
        axis on top,
        ticks=none,
        point meta min=\cmin,point meta max=\cmax,
        ]
        \addplot graphics [
        xmin=\xmin, xmax=\xmax,  ymin=\xmin, ymax=\xmax,
        ] {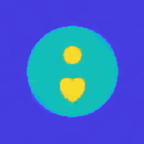};
      \end{axis}
    \end{tikzpicture}
    \caption{TV denoising of backpropagation $f_{\mathcal{Y}}$,\\ PSNR 37.8, SSIM 0.976}
  \end{subfigure}
  \begin{subfigure}[t]{.235\textwidth}
    \begin{tikzpicture}
      \begin{axis}[
        width=\modelwidth,
        height=\modelwidth,
        enlargelimits=false,
        scale only axis,
        axis on top,
        ticks=none,
        point meta min=\cmin,point meta max=\cmax,
        ]
        \addplot graphics [
        xmin=\xmin, xmax=\xmax,  ymin=\xmin, ymax=\xmax,
        ] {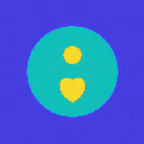};
      \end{axis}
    \end{tikzpicture}
    \caption{TV denoising of symmetr\-ized backpropagation $f_{\mathcal{Y}_\sym}$,\\ PSNR 39.5, SSIM 0.982}
  \end{subfigure}
  \begin{subfigure}[t]{.235\textwidth}
  \begin{tikzpicture}
    \begin{axis}[
      width=\modelwidth,
      height=\modelwidth,
      enlargelimits=false,
      scale only axis,
      axis on top,
      ticks=none,
      point meta min=\cmin,point meta max=\cmax,
      ]
      \addplot graphics [
      xmin=\xmin, xmax=\xmax,  ymin=\xmin, ymax=\xmax,
      ] {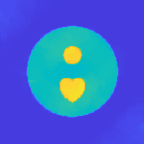};
    \end{axis}
  \end{tikzpicture}
  \caption{TV denoising of density compensation \cite{KirPot23a},\\ PSNR 36.4, SSIM 0.900}
  \end{subfigure}
  \begin{subfigure}[t]{.275\textwidth}
    \begin{tikzpicture}
      \begin{axis}[
        width=\modelwidth,
        height=\modelwidth,
        enlargelimits=false,
        scale only axis,
        axis on top,
        ticks=none,
        point meta min=\cmin,point meta max=\cmax,
        ]
        \addplot graphics [
        xmin=\xmin, xmax=\xmax,  ymin=\xmin, ymax=\xmax,
        ] {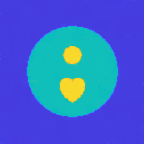};
      \end{axis}
    \end{tikzpicture}\hspace*{-5.5mm}%
    \begin{tikzpicture}
    \pgfplotscolorbardrawstandalone[colormap name=parula,
    point meta min=\cmin,point meta max=\cmax, 
    colorbar style={height=\modelwidth, width=.15cm, xshift=0mm, 
    },
    yticklabel style={
      xshift=-.5mm,yshift=0mm,font=\tiny
    } ]
    \end{tikzpicture}
    \caption{TV denoising of inverse NDFT \cite{KirQueRitSchSet21},\\ PSNR 38.9, SSIM 0.980}
  \end{subfigure}
  \caption{Reconstructions with 1\,\% Gaussian noise added to the sinogram $\utot_t$. In the second row, we post-processed the reconstructions from the first row via TV denoising \eqref{eq:tv} with $\lambda=0.02$. Otherwise, the setup is the same as in \autoref{fig:heart1-anglerot}.
  \label{fig:heart1-anglerot-noise}}
\end{figure}

The computation times on an Intel Core i7-10700 CPU with 32 GB memory are reported in \autoref{tab:time}.
As expected, the backpropagation algorithms are much faster, because they use only one adjoint NDFT whereas the inverse NDFT method uses a forward and adjoint step of the NDFT in each iteration.
The precomputation of the Banach indicatrix and the Jacobian determinant, which is independent of the data $u$, is done in reasonable time. 
Note that here we do not include the time of the precomputation step inside the NFFT library, because it is required in all four algorithms.

\begin{table}[ht]\centering
  \begin{tabular*}{.92\textwidth}{l|cccc}
     & {Backpropagation} & \shortstack{Symmetrized\cr backpropagation} & \shortstack{Density\cr compensation} & \shortstack{Inverse\cr NDFT}
    \\\hline
    Time & 11 & 11 & 11 & 202
    \\
    Precomputation & 89 & 142 & 190 & --
  \end{tabular*}
  \caption{Computation times (in ms) for \autoref{fig:heart1-anglerot}. 
    \label{tab:time}}
\end{table}

\paragraph*{Object rotation}
In our third setup, we take the fixed incidence $\bs = (0,1)$ and the rotation $R(t) = \begin{psmallmatrix} \cos t & -\sin t \\ \sin t & \cos t \end{psmallmatrix}$ for $t\in [0,3\pi/2]$ as in \autoref{fig:2d-max-cover-34}.
Here the reconstruction highly depends on the discretization of $x$ near the boundary.
Therefore we use a different grid $x_m = \cos (\pi m / M)$ for $m=1,\dots,M=160$, such that the discrete Fourier coverage $\{T(x_m,t_n)\}_{m,n=1}^{M,N}$ does not have large gaps around the origin.
The reconstructions are shown in \autoref{fig:heart1-fig4}, where we can see a significant effect of the symmetrization.
This is expected as the Fourier coverage $\mathcal{Y}$ has large gaps, see \autoref{fig:heart1-fig4-indicatrix},
but its symmetrization $\mathcal{Y}_\sym$ from \eqref{eq:Ysym} is the whole disk of radius $2 k_0$.
The visual quality of the symmetrized backpropagation is comparably to the inverse NDFT, but the error measures are somewhat worse.
Furthermore, we notice some numerical issues of the estimation of the Banach indicatrix $\operatorname{Card}(T^{-1}(\by))$ near the boundary $\abs{\by}=2\pi$ corresponding to $\abs{x} = 1$.

\begin{figure}[ht!]\centering
  \begin{subfigure}[t]{.325\textwidth}\centering
    \begin{tikzpicture}
      \begin{axis}[
        width=\modelwidth, height=\modelwidth,
        enlargelimits=false,
        scale only axis,
        axis on top,
        ticks=none,
        point meta min=\cmin,point meta max=\cmax,
        colorbar,colorbar style={
          width=.15cm, xshift=-0.5em, 
        },
        ]
        \addplot graphics [
        xmin=\xmin, xmax=\xmax,  ymin=\xmin, ymax=\xmax,
        ] {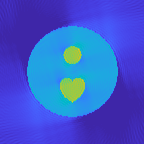};
      \end{axis}
    \end{tikzpicture}
    \caption{Backpropagation $f_{\mathcal{Y}}$ from \eqref{eq:reconstruction_dis},\\ PSNR 24.5, SSIM 0.454}
  \end{subfigure}\hfill
  \begin{subfigure}[t]{.325\textwidth}\centering
    \begin{tikzpicture}
      \begin{axis}[
        width=\modelwidth, height=\modelwidth,
        enlargelimits=false,
        scale only axis,
        axis on top,
        ticks=none,
        point meta min=.8, point meta max=4.2,
        colorbar,colorbar style={
          width=.15cm, xshift=-0.5em,   
        },
        ]
        \addplot graphics [
        xmin=-13, xmax=13,  ymin=-13, ymax=13,
        ] {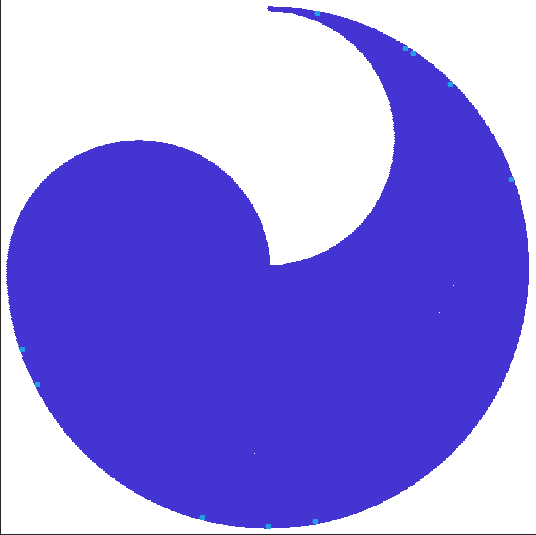};
      \end{axis}
    \end{tikzpicture}
    \caption{Fourier coverage $\mathcal Y$ for (a), the color is the Banach indicatrix $\operatorname{Card}( T^{-1}(\by))$ \label{fig:heart1-fig4-indicatrix}}
  \end{subfigure}\hfill
  \begin{subfigure}[t]{.325\textwidth}\centering
    \begin{tikzpicture}
      \begin{axis}[
        width=\modelwidth, height=\modelwidth,
        enlargelimits=false,
        scale only axis,
        axis on top,
        ticks=none,
        point meta min=\cmin,point meta max=\cmax,
        colorbar,colorbar style={
          width=.15cm, xshift=-0.5em, 
        },
        ]
        \addplot graphics [
        xmin=\xmin, xmax=\xmax,  ymin=\xmin, ymax=\xmax,
        ] {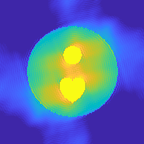};
      \end{axis}
    \end{tikzpicture}
    \caption{Density compensation \cite{KirPot23a},\\ PSNR 20.8, SSIM 0.359}
  \end{subfigure}
  
  \begin{subfigure}[t]{.325\textwidth}\centering
    \begin{tikzpicture}
      \begin{axis}[
        width=\modelwidth, height=\modelwidth,
        enlargelimits=false,
        scale only axis,
        axis on top,
        ticks=none,
        point meta min=\cmin,point meta max=\cmax,
        colorbar,colorbar style={
          width=.15cm, xshift=-0.5em, 
        },
        ]
        \addplot graphics [
        xmin=\xmin, xmax=\xmax,  ymin=\xmin, ymax=\xmax,
        ] {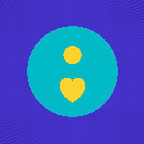};
      \end{axis}
    \end{tikzpicture}
    \caption{Backpropagation $f_{\mathcal{Y}_\sym}$ with symmetrization, PSNR 32.5, SSIM 0.410}
  \end{subfigure}\hfill
  \begin{subfigure}[t]{.325\textwidth}\centering
    \begin{tikzpicture}
      \begin{axis}[
        width=\modelwidth, height=\modelwidth,
        enlargelimits=false,
        scale only axis,
        axis on top,
        ticks=none,
        point meta min=.8, point meta max=4.2, 
        colorbar,colorbar style={
          width=.15cm, xshift=-0.5em,  
        },
        ]
        \addplot graphics [
        xmin=-13, xmax=13,  ymin=-13, ymax=13,
        ] {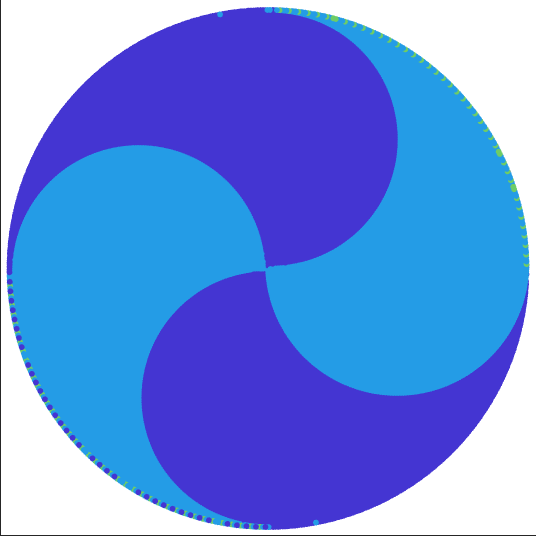};
      \end{axis}
    \end{tikzpicture}
    \caption{Banach indicatrix $\operatorname{Card}(T_\sym^{-1}(\by))$ and Fourier coverage $\mathcal{Y}_\sym$ for (d)}
  \end{subfigure}\hfill
  \begin{subfigure}[t]{.325\textwidth}\centering
    \begin{tikzpicture}
      \begin{axis}[
        width=\modelwidth, height=\modelwidth,
        enlargelimits=false,
        scale only axis,
        axis on top,
        ticks=none,
        point meta min=\cmin,point meta max=\cmax,
        colorbar,colorbar style={
          width=.15cm, xshift=-0.5em,  
        },
        ]
        \addplot graphics [
        xmin=\xmin, xmax=\xmax,  ymin=\xmin, ymax=\xmax,
        ] {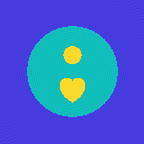};
      \end{axis}
    \end{tikzpicture}
    \caption{Inverse NDFT \cite{KirQueRitSchSet21},\\ PSNR 41.9, SSIM 0.983}
  \end{subfigure}
  \caption{Reconstructions for setup of \autoref{fig:2d-max-cover-34}.
    For comparison, we also show the inverse NDFT \cite{KirQueRitSchSet21} and the density compensation \cite{KirPot23a}.
    \label{fig:heart1-fig4}}
\end{figure}

\section{Conclusion}

In this article we have studied several questions related to diffraction tomography in $\R^d$. 
We derived a generalization of the Fourier diffraction theorem for compactly supported inhomogeneity $g\in L^1(\R^d)$ and a measurement hyperplane that may intersect $\supp g.$ 
Building on this result, we presented a novel filtered backpropagation formula, that is, an explicit expression for the $L^2$ best approximation of $f$ given the available data.
This reconstruction formula correctly handles a general experiment where a change of illumination and a rigid motion of the object occur simultaneously.
The critical quantity in the evaluation of the resulting $d$-dimensional integral is the Banach indicatrix, which can be difficult to determine exactly.
We have addressed this issue with a numerical estimation method.
Numerical tests suggest that the filtered backpropagation formula can compete with the inverse NDFT in terms of reconstruction quality, while having lower computation times.

\subsection*{Acknowledgments} 
This work is supported by the Austrian Science Fund (FWF), SFB 10.55776/F68 (``Tomography across the Scales''), and by the German Research Foundation DFG (STE 571/19-1, project number 495365311). The financial support by the Austrian Federal Ministry for Digital and Economic Affairs, the National Foundation for Research, Technology and Development and the Christian Doppler Research Association is gratefully acknowledged. This work was initiated while the third-named author was with the Johann Radon Institute for Computational and Applied Mathematics (RICAM) of the Austrian Academy of Sciences. For open access purposes, the authors have applied a CC BY public copyright license to any author-accepted manuscript version arising from this submission. 

\section*{References}
\renewcommand{\i}{\ii}
\printbibliography[heading=none]

\end{document}